\documentclass[a4paper,12pt]{article} 

\title{Enumerating odd--degree hyperelliptic \\ curves and abelian surfaces over $\mathbb{P}^1$}
\date{\vspace{-3ex}}
\author{Changho Han and Jun--Yong Park}

\usepackage{amsmath,amssymb,graphicx,amsthm,amsfonts,verbatim,mathtools,thmtools,fullpage,tikz-cd,url,tensor}
\usepackage[normalem]{ulem}
\usepackage[shortlabels]{enumitem}
\usepackage[all,cmtip]{xy}

\newtheorem{thm}{Theorem}[section]
\newtheorem{Mthm}[thm]{Main Theorem}
\newtheorem{lem}[thm]{Lemma}
\newtheorem{cor}[thm]{Corollary}
\newtheorem{prop}[thm]{Proposition}
\theoremstyle{definition}
\newtheorem{defn}[thm]{Definition}

\newtheorem{conj}[thm]{Conjecture}

\newtheorem{rmk}[thm]{Remark}
\newtheorem{note}[thm]{Note}
\newtheorem{exmp}[thm]{Example}

\newcommand{\iso}{\cong}

\newcommand{\Hg}{\overline{\mathcal{H}}}
\newcommand{\Mg}{\overline{\mathcal{M}}}

\newcommand{\Me}{\overline{\mathcal{M}}_{1,1}}

\newcommand{\pf}{\mathfrak{p}}

\newcommand{\I}{{\mathop{\rm I}}}

\newcommand{\Pcv}{\mathcal{P}(\vec{\lambda})}

\newcommand{\Pov}{\mathcal{P}(\vec{\Lambda})}

\newcommand{\Ac}{\mathcal{A}}
\newcommand{\Bc}{\mathcal{B}}

\newcommand{\Ec}{\mathcal{E}}
\newcommand{\Hc}{\mathcal{H}}
\newcommand{\Oc}{\mathcal{O}}
\newcommand{\Pc}{\mathcal{P}}
\newcommand{\Lc}{\mathcal{L}}
\newcommand{\Nc}{\mathcal{N}}
\newcommand{\M}{\mathcal{M}}
\newcommand{\Zc}{\mathcal{Z}}

\newcommand{\Fc}{\mathcal{F}}
\newcommand{\Ic}{\mathcal{I}}

\newcommand{\Lambdavec}{{\vec{\Lambda}}}

\newcommand{\Xc}{\mathcal{X}}

\newcommand{\Yc}{\mathcal{Y}}

\newcommand{\Qlb}{\overline{\Qb}_\ell}

\newcommand{\Ql}{\Qb_{\ell}}

\newcommand{\Z}{\mathbb{Z}}
\newcommand{\Q}{\mathbb{Q}}
\newcommand{\N}{\mathbb{N}}

\newcommand{\Pb}{\mathbb{P}}
\newcommand{\Ab}{\mathbb{A}}
\newcommand{\A}{\mathbb{A}}
\newcommand{\Cb}{\mathbb{C}}

\newcommand{\Fb}{\mathbb{F}}
\newcommand{\Gb}{\mathbb{G}}
\newcommand{\Lb}{\mathbb{L}}
\newcommand{\Nb}{\mathbb{N}}
\newcommand{\Qb}{\mathbb{Q}}
\newcommand{\Rb}{\mathbb{R}}
\newcommand{\Zb}{\mathbb{Z}}
\newcommand{\Fqbar}{\overline{\mathbb{F}}_q}
\newcommand{\et}{{\acute{et}}}

\newcommand{\Spec}{\mathrm{Spec}}
\newcommand{\Proj}{Proj}

\DeclareMathOperator{\Frob}{Frob}

\DeclareMathOperator{\Hom}{Hom}

\DeclareMathOperator{\Gal}{Gal}

\DeclareMathOperator{\Aut}{Aut}

\begin{document}

    \maketitle

    \vspace{-2ex}

    \begin{abstract}
    Given asymptotic counts in number theory, a question of Venkatesh asks what is the topological nature of lower order terms. We consider the arithmetic aspect of the inertia stack of an algebraic stack over finite fields to partially answer this question. Subsequently, we acquire new sharp enumerations on quasi--admissible odd--degree hyperelliptic curves over $\mathbb{F}_q(t)$ ordered by bounded discriminant height.
    \end{abstract}


    \vspace{-4ex}



    \section{Introduction}
    \label{sec:intro}

    \par In \cite[Problem 5]{GGW}, Akshay Venkatesh asks the following question:

    \vspace{-.06in}
 
    \begin{center}
    \textit{What is the topological meaning of secondary terms \\ appearing in asymptotic counts in number theory?}
    \end{center}

    \vspace{-.06in}

    \par As explained therein by Venkatesh, in many interesting number theory problems (e.g., counting number fields, arithmetic curves or abelian varieties over a number field) one has not only a main term in the asymptotic count, but a secondary term or more. For example, the number of cubic fields of discriminant up to $\Bc$ is

    \vspace{-.13in}

    \begin{equation*}\label{eq:cubic-fields}
    a\Bc + b\Bc^{5/6} + \mathrm{lower~order~terms}
    \end{equation*}

    \par We have very little understanding of these lower order terms. They are not just of theoretical interest: when one tries to verify the conjectures numerically, one finds that the secondary terms are dominant in the computational range. 


    \medskip

    \par Note that the moduli functors we wish to enumerate are often represented by algebraic stacks rather than by schemes (or algebraic spaces) due to the presence of non-trivial automorphisms of the objects we wish to parameterize. If we consider a finite field analogue, the traditional approaches to count the number of rational points on the moduli spaces do not render every lower order terms. This is because the Grothendieck-Lefschetz trace formula (relating point counts and $\ell$-adic cohomologies) for algebraic stacks as in \cite{Behrend} counts the rational points with weights (given a rational point $x$, its weight is $\frac{1}{\Aut(x)}$). Instead, we must acquire the number $|\Xc(\Fb_q)/\sim|$ of $\Fb_q$--isomorphism classes of $\Fb_q$--points of the algebraic stack $\Xc$, i.e., the non--weighted point count of $\Xc$ over $\Fb_q$. In this regard, the coarse moduli space $c: \Xc \to X$ is insufficient as $|X(\Fb_q)| \neq |\Xc(\Fb_q)/\sim|$.

    \medskip

    \par This discrepancy naturally raises the following question:

    \vspace{-.1in}

    \begin{center}
    \textit{Which arithmetic invariant of a specific geometric object $\Yc$ is equal to the \\ non--weighted point count $|\Xc(\Fb_q)/\sim|$ of the algebraic stack $\Xc$ over $\Fb_q$?}
    \end{center}

    \vspace{-.1in}

    \par We clarify the arithmetic role of the \textit{inertia stack} $\Ic(\Xc)$ of an algebraic stack $\Xc$ over $\Fb_q$ which parameterizes pairs $(x,\xi)$, where $x \in \Xc$ and $\xi$ is the conjugate class of $g \in \Aut(x)$.

    \begin{thm}\label{thm:ptcounts}
        Let $\Xc$ be an algebraic stack over $\Fb_q$ of finite type with quasi-separated finite type diagonal and let $\Ic(\Xc)$ be the inertia stack of $\Xc$. Then, 
        $$ |\Xc(\Fb_q)/\sim| = \#_q(\Ic(\Xc)) $$ 
        where $\#_q(\Ic(\Xc))$ is the weighted point count of the inertia stack $\Ic(\Xc)$ over $\Fb_q$.
    \end{thm}
    
    \par Before drawing the connection of this Theorem to Venkatesh's question, let us first consider a simpler problem instead where we want to find the non-weighted point count $|\Xc(\Fb_q)/\sim|$ of a Deligne-Mumford moduli stack $\Xc/\Fb_q$ of finite type with affine diagonal. In a given counting problem of number theory, one must be aware of the discriminant involved as the relevant moduli stack $\Xc$ is often not quasi-compact (so cannot be of finite type), but is rather a disjoint union of connected or irreducible components $\Xc_{\Bc}$ of finite type, indexed by ranges of values $0< ht(\Delta) \le \Bc$ of height of discriminant up to $\Bc$. In this regard Venkatesh's question over finite fields is then equivalent to understanding the lower order terms of the counting function $N(\Bc)$ as a function of the bounded height $\Bc$


    \begin{align*}
        &N(\Bc):=|\hat \Xc_{\Bc}(\Fb_q)/\sim| = \#_q(\Ic(\hat \Xc_{\Bc})) = \sum_{\Bc' \le \Bc} \#_q(\Ic(\Xc_{\Bc'})), \;\;\; \hat \Xc_{\Bc}:=\bigsqcup_{\Bc' \le \Bc} \Xc_{\Bc'}
    \end{align*}

    Therefore, the lower order terms of $N(\Bc)$ are determined by the growth pattern of $\#_q(\Ic(\hat \Xc_{\Bc}))$ with respect to $\Bc$. Here, we note that the geometry of $\Ic(\Xc)$ can be quite complicated. For example, even if $\Xc$ is irreducible, $\Ic(\Xc)$ can be disconnected, with many irreducible components of different dimensions corresponding to different automorphisms. Also, $\Ic(\Xc)$ may have intersecting irreducible components which are possibly singular. And crucially, $\Ic(\Xc)$ could contain lower-dimensional irreducible components (non-existent on either $\Xc$ or $X$) which will contribute to various lower order terms. Coming back to understanding the algebro-topological meaning of the lower order terms of $N(\Bc)$, we see that the weighted point count of the inertia stack $\#_q(\Ic(\hat \Xc_{\Bc}))$ over $\Fb_q$ is naturally equal to the alternating sum of trace of geometric Frobenius via the Grothendieck-Lefschetz trace formula for algebraic stacks as in Theorem~\ref{SBGLTF} by \cite{Behrend,LO,Sun}. 


    \begin{align*}
        N(\Bc)=\sum\limits_{i=0}^{2\dim\Ic(\hat \Xc_{\Bc})}~(-1)^i \cdot tr\big( \Frob^*_q : H^i_{\et, c}(\Ic(\hat \Xc_{\Bc})_{/\Fqbar};\Ql) \to H^i_{\et, c}(\Ic(\hat \Xc_{\Bc})_{/\Fqbar};\Ql) \big)
    \end{align*}

    It is standard to consider the natural grading determined by degree $i$ of compactly-supported cohomologies; then, the top degree cohomology (when $i=2\dim\Ic(\Xc)$) can be interpreted as the main leading term and the rest of the lower order terms of $N(\Bc)$ corresponds to the lower degree, compactly-supported, $\ell$-adic cohomologies of $\Ic(\hat \Xc_{\Bc})$ with geometric Frobenius weights. A priori, however, the general mechanism that precisely determines which connected component(s) of $\Ic(\hat \Xc_{\Bc})$ contribute(s) to a given lower order term of a specific order remains unclear without fixing the counting/moduli problem $\hat \Xc_{\Bc}$ and studying the arithmetic geometry of $\Ic(\hat \Xc_{\Bc})$ with regard to $\#_q(\Ic(\hat \Xc_{\Bc})) = |\hat \Xc_{\Bc}(\Fb_q)/\sim|$. 

    \medskip

    In essence, this analysis, which is an extension of the framework of the Weil conjectures to the rational points of inertia stacks of arithmetic moduli stacks, provides a partial answer to the question of Venkatesh through the lower-dimensional irreducible components of $\Ic(\Xc)$ corresponding to different conjugacy classes of automorphisms as in Definition~\ref{def:decomp_conjugate}.

    \medskip

    \par Due to the inherent complexity of inertia stacks in general, we instead focus on irreducible algebraic stacks $\Xc$ of finite type (with conditions on the diagonal). Furthermore, we restrict to the case when $\Xc \cong [U/G]$ is a quotient stack, which is a testing ground for the strategy above. Then, the inertia stack $\Ic(\Xc)$ turns out to be a quotient stack as well, of the form $[R_{\Delta}/G]$ (see Corollary~\ref{cor:presentation_inertia}). If $\Xc$ is furthermore Deligne-Mumford with affine diagonal, then $\Ic(\Xc)$ decomposes into a disjoint union of $\Xc$ and other components, which are fixed loci of nontrivial elements of $G$ (see \eqref{eq:decomp_order} and Definition~\ref{def:decomp_conjugate} for more details). By using the idea of cut-and-paste by Grothendieck in $K_0(\mathrm{Stck}_K)$ (this is a natural generalization of the Grothendieck ring of varieties, see Definition~\ref{defn:GrothringStck}), we acquire the motive $\{\Ic(\Xc)\}$ which renders $\#_q(\Ic(\Xc))$ to be a polynomial in $q$ through the decomposition of $\Ic(\Xc)$ whenever every piece is reasonably simple (For instance, when $G=\Gb_m$ then the motive $\{\Ic(\Xc)\}=\{R_{\Delta}\}/\{G\}$ by Lemma~\ref{lem:Gm_quot}).

    \medskip

    \par As an application of this strategy, we consider the Hom stack $\Hom_n(\Pb^1,\Pcv)$ parameterizing the degree $n \in \Zb_{\geq 1}$ morphisms $f:\Pb^1 \rightarrow \Pcv$ of rational curves on a weighted projective stack $\Pcv$ (see Definition~\ref{def:wtproj}) with $f^*\Oc_{\Pcv}(1) \simeq \Oc_{\Pb^1}(n)$. Since the Hom stacks are quotient stacks by Remark~\ref{rmk:Hom_substack_wtproj}, the strategy works out nicely and we prove that the exact weighted point count $\#_q\left(\Ic\left(\Hom_n(\Pb^1,\Pcv)\right)\right)$ of the inertia stack over $\Fb_q$ can be acquired to be a polynomial in $q$ which in turn provides the exact non--weighted point count $\left|\Hom_n(\Pb^1,\Pcv)(\Fb_q)/\sim\right|$ of the Hom stack over $\Fb_q$ by Theorem~\ref{thm:ptcounts}. Hom stacks are important classes of Deligne-Mumford stacks as numerous arithmetic moduli problems can naturally be approximated (if not identified) to weighted projective stacks or Hom stacks under mild condition on the characteristic of the base field $K$. For example, both authors in \cite{HP} showed that $\Lc_{1,12n} := \Hom_n(\Pb_{\Fb_q}^1,\Pc_{\Fb_q}(4,6))$ represents the moduli stack of stable elliptic fibrations over $\Pb^1_{\Fb_q}$ with discriminant degree $12n$ (as $(\Me)_{\Fb_q} \iso \Pc_{\Fb_q}(4,6)$ is the moduli stack of stable elliptic curves when $2,3 \nmid q$), and computed the exact non-weighted point count $|\Lc_{1,12n}(\Fb_q)/\sim|$ over $\Fb_q$ by the motive $\{\Lc_{1,12n}\} \in K_0(\mathrm{Stck}_K)$ (see also \cite{PS}).

    \medskip
    
    \par For the moduli stack of genus $g \ge 2$ fibrations over $\Pb^1_{\Fb_q}$, however, it is difficult to acquire the arithmetic invariants of $\Ic\left(\Hom(\Pb^1, \Mg_g)\right)$ due to the global geometry of the Deligne--Mumford moduli stack $\Mg_g$ of stable genus $g$ curves formulated in \cite{DM}. For example, the coarse moduli space $\overline{M}_g$ is of general type for $g \ge 24$ by the fundamental works of Harris, Mumford and Eisenbud in \cite{HM, EH} which in turn makes the study of (rational) curves on $\Mg_{g \ge 24}$ ineffective for counting stable curves of sufficiently high genus over $\Pb^1_{\Fb_q}$. Instead, we consider the following strategy:
 
    \begin{center}
    \textit{Could we approximate $\Mg_g$ by $\Pc(\vec{\lambda_g})$ and show that the non-weighted point count of the Hom stack $\Hom_{n}(\Pb^1,\Pc(\vec{\lambda_g}))$ is an upper bound for the non-weighted point count of the Moduli stack $\Hom_{n}(\Pb^1,\Mg_g)$ of stable genus $g \ge 2$ fibrations over $\Pb^1_{\Fb_q}$?}
    \end{center}

    \par Remarkably, the strategy can be executed successfully if we restrict to hyperelliptic genus $g \ge 2$ curves\footnote{~However, similar to elliptic curves, due to the presence of generic non-trivial automorphism of \textit{hyperelliptic involution} the fine moduli space for hyperelliptic curves does not exist and we must work with the fine moduli \textit{stack} especially for the existence of the universal family of hyperelliptic curves.}. Firstly, all smooth genus 2 curves are hyperelliptic, thus $\M_2 \iso \Hc_2$. In general, recall that an \textit{odd--degree} hyperelliptic curve has a marked rational Weierstrass point at $\infty$. In this paper, we will concentrate on the moduli substack $\Hc_{g,\underline{1}} \subset \M_{g,1}$ of hyperelliptic genus $g \ge 2$ curves with $1$ marked rational Weierstrass point (which has the same dimension as $\Hc_{g}$) as we focus on counting odd--degree hyperelliptic genus $g \ge 2$ curves. Since $\Hc_{g,\underline{1}}$ is not proper, we consider the proper moduli stack $\Hg_{g,\underline{1}}:=\overline{\Hc_{g,\underline{1}}} \subset \Mg_{g,1}$ (meaning the reduced closure) of stable odd--degree hyperelliptic curves. Similar to $\Mg_g$, extracting the exact arithmetic invariants of $\Hom_n(\Pb^1, \Hg_{g,\underline{1}})$ is challenging, so we consider (upto some conditions on characteristic of $\Fb_q$) a different extension of smooth odd--degree hyperelliptic curves such that the compactified moduli stack is a weighted projective stack, originally introduced as a special case of \cite[Definition 2.5]{Fedorchuk}:
    
    \begin{defn}\label{def:qadm_hyp}
        Fix an integral reduced $K$-scheme $B$, where $\mathrm{char}(K) \neq 2$. A flat family $u:C \rightarrow B$ of genus $g \ge 2$ curves is \emph{quasi--admissible} if every geometric fiber has at worst $A_{2g-1}$-singularity (i.e., \'etale locally defined by $x^2+y^{m}$ for some $0< m \le 2g$), and factors through a separable morphism $\phi:C \to H$ of degree 2 where $H$ is a $\Pb^1$-bundle over $B$ with a distinguished section (often called $\infty$) which is a connected component of the branch locus of $u$.
    \end{defn}
    
    \par The notion of quasi--admissible covers whereby the general member of $C$ is \textit{not} an admissible cover of $\Pb^1$ is natural and have been studied in depth by \cite[\S 2.4.]{Stankova} as the closest covers to the original families of stable curves. For example, if $\mathrm{char}(K) > 2g+1$ or $0$, then a quasi--admissible curve over any $K$-scheme $B$ can be written as an odd--degree hyperelliptic curve via generalized Weierstrass equation:
    \begin{equation}\label{eq:qadm_Weier}
    y^2 = f(x) = x^{2g+1} + a_{4}x^{2g-1} + a_{6}x^{2g-2} + a_{8}x^{2g-3} + \cdots + a_{4g+2},
    \end{equation}
    where $a_i$'s are appropriate sections of suitable line bundles on $B$ where not all of them simultaneously vanish at anywhere on $B$. Here, we identify the section at $\infty$ as the locus missed by the above affine equation. This identification is a consequence of Proposition~\ref{prop:moduli_qadm_curve}, where we show that the Deligne--Mumford moduli stack $\mathcal{H}_{2g}[2g-1]$ of quasi--admissible curves of genus $g$ is isomorphic to the weighted projective stack $\Pc(\vec{\lambda_g})$ for $\vec{\lambda_g} := (4,6,8,\dotsc,4g+2)$ over base field $K$ with $\mathrm{char}(K)=0$ or $> 2g+1$. Assigning $\Hc_{2g}[2g-1]$ as the target stack which naturally carries the universal family, we can now formulate the moduli stack $\Lc_{g}$ of quasi--admissible hyperelliptic genus $g$ fibrations with a marked Weierstrass section.
    
    \begin{prop}\label{prop:moduli_qadm_fib}
        Assume $\mathrm{char}(K) = 0$ or $> 2g+1$. Then, the moduli stack $\Lc_{g}$ of quasi--admissible odd--degree hyperelliptic genus $g$ fibrations over $\Pb^{1}$ with a marked Weierstrass section is the tame Deligne--Mumford stack $\Hom_{>0}(\Pb^1,\Hc_{2g}[2g-1])$ parameterizing the $K$-morphisms $f:\Pb^1 \rightarrow \Hc_{2g}[2g-1]$ with $\deg f^*\Oc_{\Hc_{2g}[2g-1]}(1)>0$.
    \end{prop}

    Above proposition shows that $\Lc_{g}$ is a well-behaving object parametrizing quasi--admissible curves of genus $g$ over $\Pb^1_K$. The proposition below signifies the importance of this stack in regard to understanding the moduli of stable odd--degree hyperelliptic genus $g$ curves over $\Pb^1_K$ (with smooth generic fiber):

    \begin{thm} \label{thm:hypell_surf_and_qadm_fib}
      Fix a base field $K$ with $\mathrm{char}(K) > 2g+1$. Then there is a canonical fully faithful functor of groupoids $\mathcal F: \mathcal{S}_g(K) \rightarrow \Lc_g(K)$ from the groupoid $\mathcal{S}_g(K)$ of stable odd--degree hyperelliptic genus $g \ge 2$ curves over $\Pb^1_{\Fb_q}$ with a marked Weierstrass point and generically smooth fibers to $\Lc_g(K)$.
    \end{thm}


    \par To effectively count the non--weighted $\Fb_q$--points of the moduli stack $\Lc_g$, we need to impose a notion of \textit{bounded height} on those $\Fb_q$--points. Thanks to the works of Lockhart and Liu, we have a natural definition (see Definition~\ref{Hyp_Disc}) of a hyperelliptic discriminant $\Delta_g$ of quasi--admissible curves as in \cite{Lockhart, Liu2}. It is a homogeneous polynomial of degree $4g(2g+1)$ on variables $a_i$'s, where each $a_i$ has degree $i$ ($a_i$'s are as in equation~\eqref{eq:qadm_Weier} where $B=\Pb^1_{\Fb_q}$ in this case). Moreover, since $\Pc(4,6,8,\dotsc,4g+2)$ carries a primitive ample line bundle $\Oc_{\Pc(4,6,8,\dotsc,4g+2)}(1)$, the degree of the discriminant $\Delta_g$ of a given quasi--admissible fibration $f:\Pb^1 \rightarrow \mathcal{H}_{2g}[2g-1] \cong \Pc(4,6,8,\dotsc,4g+2)$ is equal to $4g(2g+1)n$ where $f^*\Oc_{\Pc(4,6,8,\dotsc,4g+2)}(1) \cong \Oc_{\Pb^1}(n)$. Therefore, Hom stack $\Hom_n(\Pb^1,\Pc(4,6,8,\dotsc,4g+2))$ parameterizing such morphisms is the moduli stack $\Lc_{g, |\Delta_{g}| \cdot n}$ of quasi--admissible genus $g \ge 2$ fibrations of a fixed discriminant degree $|\Delta_{g}| \cdot n =4g(2g+1)n$. Consequently, we acquire the exact weighted point count $\Ic\left(\Lc_{g, |\Delta_{g}| \cdot n}\right)$ over $\Fb_q$ which is equal to the exact non--weighted point count $\left|\Lc_{g, |\Delta_{g}| \cdot n}(\Fb_q)/\sim\right|$ over $\Fb_q$ by Theorem~\ref{thm:ptcounts}.

    \begin{thm} \label{thm:low_genus_count} 
    If $\mathrm{char}(\Fb_q)> 2g+1$, the number $\left|\Lc_{g, |\Delta_{g}| \cdot n}(\Fb_q)/\sim\right|$ of $\Fb_q$--isomorphism classes of quasi--admissible odd--degree hyperelliptic genus $g$ fibrations over $\Pb^1_{\Fb_q}$ with a marked Weierstrass point and hyperelliptic discriminant of degree $|\Delta_{g}| \cdot n = 4g(2g+1)n$ is equal to 
        \begingroup
        \allowdisplaybreaks
        \begin{align*}
            \left|\Lc_{2,40n}(\Fb_q)/\sim\right|& = 2 \cdot q^{28n} \cdot (q^{3} + q^{2} + q^{1} - q^{-1} - q^{-2} - q^{-3}) + \delta(4,q-1) \cdot 2 \cdot q^{12n} \cdot (q^{1}-q^{-1})\\\\
            \left|\Lc_{3,84n}(\Fb_q)/\sim\right|& = 2 \cdot q^{54n} \cdot (q^{5} +  \dotsb + q^{1} - q^{-1}  - \dotsb - q^{-5})\\
            &\phantom{=}\:+ \delta(4,q-1) \cdot 2 \cdot q^{24n} \cdot (q^{2} + q^{1} - q^{-1} - q^{-2}) + \delta(6,q-1) \cdot 4 \cdot q^{18n} \cdot (q^{1} - q^{-1})\\\\
            \left|\Lc_{4,144n}(\Fb_q)/\sim\right|& = 2 \cdot q^{88n} \cdot (q^{7} + \dotsb + q^{1} - q^{-1} - \dotsb - q^{-7}) \\
            &\phantom{=}\: + \delta(4,q-1) \cdot 2 \cdot q^{40n} \cdot (q^{3} + q^{2} + q^{1} - q^{-1} - q^{-2} - q^{-3})\\
            &\phantom{=}\: + \delta(6,q-1) \cdot 4 \cdot q^{36n} \cdot (q^{2} + q^{1} - q^{-1} - q^{-2}) + \delta(8,q-1) \cdot 4 \cdot q^{24n} \cdot (q^{1} - q^{-1})
        \end{align*}
        \endgroup
    where
    \[
        \delta(a,b):=
        \begin{cases}
            1 & \text{if } a \text{ divides } b,\\
            0 & \text{otherwise.}
        \end{cases}
    \]
    For genus $g \ge 5$, the corresponding exact non--weighted point count $\left|\Lc_{g, |\Delta_{g}| \cdot n}(\Fb_q)/\sim\right|$ of the moduli stack $\Lc_{g, |\Delta_{g}| \cdot n}$ over $\Fb_q$ can be similarly worked out.
    \end{thm}

    \par Given an odd--degree hyperelliptic genus $g \ge 2$ curve $X$ over $\Pb^1_{\Fb_q}$ with $\text{char} (\Fb_q) > 2g+1$, define the \emph{height} of the hyperelliptic discriminant $\Delta_g(X)$ to be $ht(\Delta_g(X)):=q^{\deg \Delta_g(X)}=q^{4g(2g+1)n}$ (see Definition~\ref{def:height_Disc}). Then, we define the counting function $\Zc_{g, \Fb_q(t)}(\Bc)$ as
    \[\Zc_{g, \Fb_q(t)}(\Bc) := |\{\text{Quasi--admissible odd--degree hyperelliptic curves over } \Pb^1_{\Fb_q} \text{ with } 0<ht(\Delta_{g}) \le \Bc\}|\] 

    
    \smallskip

    We acquire the following sharp enumerations via Theorem~\ref{thm:low_genus_count} in \S\ref{sec:globfield}

    \begin{Mthm} [Sharp enumeration on $\Zc_{g, \Fb_q(t)}(\Bc)$] \label{Mthm:g_sharper_asymp}
        If $\text{char} (\Fb_q) > 2g+1$, then the function $\Zc_{g, \Fb_q(t)}(\Bc)$, which counts the number of quasi--admissible odd--degree hyperelliptic genus $g \ge 2$ curves $X$ over $\Pb^1_{\Fb_q}$ ordered by $0< ht(\Delta_{g}(X)) = q^{4g(2g+1)n} \le \Bc$, satisfies:
        \begingroup
        \allowdisplaybreaks
        \begin{align*}
         \Zc_{2, \Fb_q(t)}(\Bc) &\le a_{2,2} \cdot  \Bc^{\frac{7}{10}} + a_{2,4} \cdot  \Bc^{\frac{3}{10}} + b_{2,4} \\
         \Zc_{3, \Fb_q(t)}(\Bc) &\le a_{3,2} \cdot  \Bc^{\frac{9}{14}} + a_{3,4} \cdot  \Bc^{\frac{2}{7}} + a_{3,6} \cdot \Bc^{\frac{3}{14}} + b_{3,6} \\
         \Zc_{4, \Fb_q(t)}(\Bc) &\le a_{4,2} \cdot  \Bc^{\frac{11}{18}} +  a_{4,4} \cdot \Bc^{\frac{5}{18}} + a_{4,6} \cdot \Bc^{\frac{1}{4}} + a_{4,8} \cdot \Bc^{\frac{1}{6}} + b_{4,8}
        \end{align*}
        \endgroup
        which is an equality when $\Bc = q^{4g(2g+1)n}$ with $n \in \mathbb{Z}_{\geq 1}$ implying that each upper bound is the sharp enumeration, i.e., the upper bound is equal to the function at infinitely many values of $\Bc \in \Zb_{\geq 1}$. Also $\delta(a,b)$ is as in Theorem~\ref{thm:low_genus_count}, and for each $g,m \in \N_{\ge 2}$, $a_{g,2}(q), a_{g,2m}(q,\delta(2m,q-1))$, and $b_{g,2m}(q,\delta(4,q-1),\delta(6,q-1),\dotsc,\delta(2m,q-1))$ are explicit rational functions as in Theorem~\ref{thm:g_sharper_asymp_full}.
    \end{Mthm} 

    \par For higher genus $g \ge 5$, the sharp enumeration on $\Zc_{g, \Fb_q(t)}(\Bc)$ rendering a closed-form formula with precise lower order terms can be similarly worked out through Theorem~\ref{thm:g_sharper_asymp_full}.

    \begin{note}\label{note:lower_order}
    When $g=2$ and $\delta(4,q-1)=1$, the secondary term $a_{2,4} \cdot  \Bc^{\frac{3}{10}}$ of $\Zc_{2, \Fb_q(t)}(\Bc)$ is induced purely from the lower dimensional irreducible component $\Hom_n(\Pb^1,\Pc(4,8)) \subsetneq \Ic\left(\Hom_n(\Pb^1,\Pc(4,6,8,10))\right)$ corresponding to $\mu_4$ generic stabilizer (see Proposition~\ref{prop:algo_ptcount} and Theorem~\ref{thm:g_sharper_asymp_full}). 
    We also note that the lower order term $b_{2,4}$ of zeroth order is a sum of contributions from both the higher dimensional irreducible component $\Hom_n(\Pb^1,\Pc(4,6,8,10)) \subsetneq \Ic\left(\Hom_n(\Pb^1,\Pc(4,6,8,10))\right)$ and the lower dimensional irreducible component $\Hom_n(\Pb^1,\Pc(4,8)) \subsetneq \Ic\left(\Hom_n(\Pb^1,\Pc(4,6,8,10))\right)$.
    The emergence of lower order terms illustrated here continues similarly for higher genus $g \ge 3$ case $\Zc_{g, \Fb_q(t)}(\Bc)$ with $\delta(r,q-1)$ for each irreducible component with $\mu_r$ generic stabilizer. 


    \end{note}

    As we have seen by Theorem~\ref{thm:hypell_surf_and_qadm_fib}, whenever $\mathrm{char}(\Fb_q)>2g+1$, counting the number $\Zc_{g,\Fb_q(t)}(\Bc)$ of quasi--admissible odd--degree hyperelliptic genus $g \ge 2$ curves over $\Pb^1_{\Fb_q}$ renders an upper bound for counting the number $\Zc'_{g,\Fb_q(t)}(\Bc)$ of stable odd hyperelliptic genus $g \ge 2$ curves over $\Pb^1_{\Fb_q}$ with generically smooth fibers. That is,

    \begin{equation}
    \Zc'_{g,\Fb_q(t)}(\Bc) \le \Zc_{g,\Fb_q(t)}(\Bc)
    \end{equation}

    \par Using this, we obtain another application regarding the enumeration of abelian varieties of dimension 2, i.e., \textit{abelian surfaces} over global function fields. By the local (i.e., infinitesimal) Torelli theorem in \cite[Theorem 2.6 and 2.7]{OS} and \cite[Theorem 12.1]{Milne}, the Torelli map $\tau_2 : \M_2 \hookrightarrow \Ac_2$, which sends a smooth projective genus 2 curve $X$ defined over a field $K$ to its principally polarized Jacobian $(\mathrm{Jac}(X), \lambda_{\theta})/K$ (where $\lambda_{\theta}$ is the theta divisor of $\mathrm{Jac}(X)$), is an open immersion. Furthermore, it is shown in \cite[4. Theorem]{OU} (see also \cite[Satz 2]{Weil}) that given a principally polarized abelian surface $(A,\lambda)$ over a field $K$, after a finite extension of scalars, is isomorphic to the canonically principally polarized (generalized) Jacobian variety $(\mathrm{Jac}(X), \lambda_{\theta})$ of a stable genus 2 curve $X$. Recall that if a curve $X$ has good reduction at a place $v \in S$ then so does its Jacobian $\mathrm{Jac}(X)$. 

    \begin{thm} [Estimate on $\Nc_{2, \Fb_q(t)}(\Bc)$]\label{thm:Ab_Surf_Fqt}
    If $\text{char} (\Fb_q) \neq 2,3,5$, then the function $\Nc_{2, \Fb_q(t)}(\Bc)$, which counts the number of principally polarized abelian surfaces $A = \mathrm{Jac}(X)$ where $X$ is a stable genus 2 curve with a marked Weierstrass point over $\Pb^1_{\Fb_q}$ ordered by $0<ht(\Delta_{2}(X)) = q^{40n} \le \Bc$, satisfies:
        \begingroup
        \allowdisplaybreaks
        \begin{align*}
        &\Nc_{2, \Fb_q(t)}(\Bc) \le a_{2,2} \cdot  \Bc^{\frac{7}{10}} + a_{2,4} \cdot  \Bc^{\frac{3}{10}} + b_{2,4} \\\\
        &a_{2,2}(q) = 2 \cdot \frac{(q^{31} + q^{30} + q^{29} - q^{27} - q^{26} - q^{25})}{(q^{28}-1)}, \;\; a_{2,4}(q) = \delta(4,q-1) \cdot 2 \cdot \frac{(q^{13} - q^{11})}{(q^{12}-1)}\\\\
        &b_{2,4}(q) = -2 \cdot \frac{(q^{31} + q^{30} + q^{29} - q^{27} - q^{26} - q^{25})}{(q^{28}-1)} - \delta(4,q-1) \cdot 2 \cdot \frac{(q^{13} - q^{11})}{(q^{12}-1)}
        \end{align*}
        \endgroup
    where
    \[
        \delta(4,q-1):=
        \begin{cases}
            1 & \text{if } 4 \text{ divides } q-1,\\
            0 & \text{otherwise.}
        \end{cases}
    \]
    \end{thm}

    \begin{proof}
    Main Theorem~\ref{Mthm:g_sharper_asymp} combined with Theorem~\ref{thm:hypell_surf_and_qadm_fib} provides an explicit upper bound on the number of stable genus 2 curves with a marked Weierstrass point over $\Pb^1_{\Fb_q}$ with $\text{char} (\Fb_q) \neq 2,3,5$.
    The upper bound follows from the properties of the Torelli map $\tau_2$ as all principally polarized abelian surfaces are isomorphic to Jacobians of genus 2 curves of compact type (i.e., genus 2 curve with dual graph equal to tree) (c.f., \cite[Theorem 2.6 and 2.7]{OS} \& \cite[4. Theorem]{OU}).
    \end{proof}

    See Theorem~\ref{thm:Hyp_Jac_Fqt} for enumerations of higher genus $g \ge 3$ hyperelliptic Jacobians.

    \medskip

    \subsection*{Organization}\label{subsec:Outline}
    \par In \S \ref{sec:Inertia_Stacks}, we establish the arithmetic geometric properties of the inertia stack $\Ic(\Xc)$ of an algebraic stack $\Xc$ thereby proving the Theorem~\ref{thm:ptcounts} and describing various decompositions of the inertia stacks of quotient stacks. In \S \ref{sec:ModuliofHom}, we formulate the Hom stack $\Hom_n(\Pb^1,\Pcv)$ of rational curves on a weighted projective stack $\Pcv$ and provide a clear decomposition of the inertia stack $\Ic(\Hom_n(\Pb^1,\Pcv))$ (i.e., each summand is the Hom stack $\Hom_n(\Pb^1,\Pc(\vec\lambda_{I_g}))$). In \S \ref{sec:count}, we use the Grothendieck ring of $K$--stacks $K_0(\mathrm{Stck}_K)$ to acquire the motive $\left\{\Hom_n(\Pb^1,\Pcv)\right\}$ (Proposition~\ref{prop:motivecount_Hom}) which provides the class $\left\{ \Ic\left(\Hom_n(\Pb_K^1,\Pc_K(\vec\lambda))\right) \right\}$. We also give an algorithm for computing $|\Hom_n(\Pb^1,\Pcv)(\Fb_q)/\sim|$. Afterwards in \S \ref{sec:qadm_fib}, we formulate the moduli stack $\Lc_{g, |\Delta_{g}| \cdot n}$ of quasi--admissible hyperelliptic genus $g$ fibrations over $\Pb^{1}$ with the hyperelliptic discriminant $\Delta_g$ via the birational geometry of surfaces. We use birational geometry to prove Theorem~\ref{thm:hypell_surf_and_qadm_fib}. Then we compute the related non--weighted point count of the moduli stack $\Lc_{g, |\Delta_{g}| \cdot n}$ over $\Fb_q$ there, proving Theorem~\ref{thm:low_genus_count}. In \S \ref{sec:globfield}, we finally establish the sharp enumerations with precise lower order terms thereby proving Main Theorem~\ref{Mthm:g_sharper_asymp}.


    
    \subsection*{Notation and conventions}\label{subsec:Conv}
    
    \par In the present paper, schemes/stacks are assumed to be defined over a field $K$, if $K$ is not mentioned explicitly or if such scheme is obviously not defined over any field (e.g., $\Spec \; \Zb$). Given a point $x$ of a scheme/stack, $\kappa(g)$ means the field of definition of $x$ (i.e., the residue field). Given a group scheme $G$ defined over a field $K$, then $Cl(G)$ is the set of conjugate classes of closed points $g$ of $G$ (here, $\kappa(g)$ is not necessarily $K$); this in general is a strictly larger set than the conjugacy class $Cl(G(K))$ of the group of $K$-rational points of $G$.

    
    \par Here, we use the convention in \cite[\S 8]{Olsson2} that the diagonal of an algebraic stack is representable (by algebraic spaces). 
    For any $T$-point $x$ of a stack $\Xc$, $\Aut(x)$ is the group of automorphisms of $x \in \Xc(T)$ (defined over $T$). We denote $\underline{\Aut}_x$ to be the automorphism space (as an algebraic space) of $x \in \Xc$ (see \eqref{eq:aut_fibprod} in $\S 2$).


    \par We identify the Weil divisors and the associated divisorial sheaves implicitly (e.g., if $X$ is a Cohen-Macaulay scheme, then the canonical divisor $K_X$ corresponds to the dualizing sheaf $\omega_X \cong \Oc(K_X)$ of $X$). Given a finite morphism $f: X \rightarrow Y$ of reduced equidimensional schemes, a branch divisor of $f$ on $Y$ means the pushforward of the ramification divisor of $f$ on $X$. Given a morphism $f: X \rightarrow Y$ of schemes with an isolated subset $Z \subset Y$ (i.e., $Y$ as a topological space is $Z \sqcup (Y\setminus Z)$ under the Zariski topology), the preimage of $Z$ in $X$ refers to the components of $X$ with their image supported on $Z$.

    \medskip


    \section{Arithmetic geometry of the inertia stack $\Ic(\Xc)$ of an algebraic stack $\Xc$}\label{sec:Inertia_Stacks}

    \par In this section, we describe the geometry of inertia stacks associated to algebraic stacks; particularly, we first recall various key properties of inertia stacks. By using these properties, we prove Theorem~\ref{thm:ptcounts}. Then we describe the groupoid structure of inertia stacks, in particular, inertia stacks of quotient stacks. For general reference on algebraic stacks, we refer the reader to \cite{Olsson2, Stacks}.

    \medskip
    
    \par Given an algebraic stack $\Xc$ defined over a field $K$, its inertia stack $\Ic(\Xc)$ is defined as:
    \begin{enumerate}
        \item objects: $(x,\alpha)$ where $x \in \Xc(T)$ for some scheme $T$ (i.e., $x: T \rightarrow \Xc$) and $\alpha \in \Aut(x)$
        \item morphisms: $\psi: (x, \alpha) \rightarrow (y, \beta)$ is given by $\phi: x \rightarrow y$ in $\mathrm{Mor}(\Xc)(T)$ such that $\phi \circ \alpha = \beta \circ \phi$, i.e., $\beta=\phi \circ \alpha \circ \phi^{-1}$
    \end{enumerate}
    Also, $\Ic(\Xc)$ is characterized by the following Cartesian diagram (by \cite[Definition 8.1.17]{Olsson2}):
    \begin{equation}\label{eq:inertia_fibprod}
        \begin{tikzcd}
            \Ic(\Xc) \arrow[d] \arrow[r] & \Xc \arrow[d, "\Delta"]\\
            \Xc  \arrow[r,"\Delta"] & \Xc \times \Xc
        \end{tikzcd}
    \end{equation}
    Note that if the representable morphism $\Delta$ satisfies a property (such as finite type, quasi-separated, etc.,), then this property is also satisfied for the representable morphism $\Ic(\Xc) \rightarrow \Xc$. In particular, $\Ic(\Xc)$ is a $\Xc$-algebraic space, i.e., $\Ic(\Xc) \times_{\Xc} T$ is an algebraic space for any morphism $T \rightarrow \Xc$ from a scheme $T$.
    
    \medskip
    
    \par To understand $\Ic(\Xc) \rightarrow \Xc$, we first pay attention to $\Delta$. Given an object $x: T \rightarrow \Xc$ of $\Xc$ from a scheme $T$, recall that the automorphism space $\underline{\Aut}_x$ of $x$ is defined to be the fiber product $\Xc \tensor[_{\Delta}]{\times}{_{x \times x}} T$. This means that $S$-points of $\underline{\Aut}_x$ are characterized by pairs $(s,\alpha)$ of maps $s: S \rightarrow T$ and automorphisms $\alpha: s^*x \rightarrow s^*x$ in the groupoid $\Xc(T)$. Since $x \times x$ factors through $\Delta$, $\underline{\Aut}_x$ fits into the following Cartesian diagram:
    \begin{equation}\label{eq:aut_fibprod}
        \begin{tikzcd}
            \underline{\Aut}_x \arrow[d] \arrow[r] & \Ic(\Xc) \arrow[d]\\
            T  \arrow[r,"x"] & \Xc
        \end{tikzcd}
    \end{equation}
    As before, representability of $\Delta$ implies that $\underline{\Aut}_x \rightarrow T$ is a morphism of algebraic spaces, and the group algebraic space structure on $\underline{\Aut}_x$ lift, realizing $\Ic(\Xc)$ as a group algebraic space over $\Xc$.

    \medskip

    \par Before proving Theorem~\ref{thm:ptcounts}, we recall the definition of a weighted point count of an algebraic stack $\Xc$ over $\Fb_q$:

    \begin{defn}\label{def:wtcount}
        The weighted point count of $\Xc$ over $\Fb_q$ is defined as a sum:
        \[
        \#_q(\Xc):=\sum_{x \in \Xc(\Fb_q)/\sim}\frac{1}{|\mathrm{Aut}(x)|},
        \]
        where $\Xc(\Fb_q)/\sim$ is the set of $\Fb_q$--isomorphism classes of $\Fb_q$--points of $\Xc$ (i.e., the set of non--weighted points of $\Xc$ over $\Fb_q$), and we take $\frac{1}{|\mathrm{Aut}(x)|}$ to be $0$ when $|\mathrm{Aut}(x)|=\infty$.
    \end{defn}
    
    \par A priori, the weighted point count can be $\infty$, but when $\Xc$ is of finite type, then the stratification of $\Xc$ by schemes as in \cite[Proof of Lemma 3.2.2]{Behrend} implies that $\Xc(\Fb_q)/\sim$ is a finite set, so that $\#_q(\Xc)<\infty$.
    
    \medskip

    \par We also recall the \textit{Grothendieck-Lefschetz trace formula} for Artin stacks by \cite{Behrend,LO,Sun}.

    \begin{thm}[Theorem 1.1. of \cite{Sun}] \label{SBGLTF} 
    Let $\Xc$ be an Artin stack of finite type over $\Fb_q$. Let $\Frob_q$ be the geometric Frobenius on $\Xc$. Let $\ell$ be a prime number different from the characteristic of $\Fb_q$, and let $ \iota : \Qlb \overset{\sim}{\to} \Cb$ be an isomorphism of fields. For an integer $i$, let $H^i_{\et, c}(\Xc_{/\Fqbar};\Qlb)$ be the cohomology with compact support of the constant sheaf $\Qlb$ on $\Xc$ as in \cite{LO}. Then the infinite sum regarded as a complex series via $\iota$
            \begin{equation}
            \sum_{i \in \Zb}~(-1)^i \cdot tr\big( \Frob^*_q : H^i_{c}(\Xc_{/\Fqbar};\Qlb) \to H^i_{c}(\Xc_{/\Fqbar};\Qlb) \big)
            \end{equation}
    is absolutely convergent to the weighted point count $\#_q(\Xc)$ of $\Xc$ over $\Fb_q$.
    \end{thm}

    \par When the stack $\Xc$ is a Deligne--Mumford stack of finite type over $\Fb_q$ with affine diagonal, then the corresponding compactly-supported, $\ell$-adic \'etale cohomology for prime number $\ell$ invertible in $\Fb_q$ is finite dimensional as a $\Ql$-algebra, making the above trace formula to hold in $\Ql$-coefficients. 

    \medskip

    \par We are now ready to prove Theorem~\ref{thm:ptcounts}:

    \begin{proof}[Proof of Theorem~\ref{thm:ptcounts}]
        Choose any $x \in \Xc(\Fb_q)/\sim$. Then the morphism $x: \Fb_q \rightarrow \Xc$ factors through a representable morphism $\overline{x}: [\Spec (\Fb_q)/\underline{\Aut}_x] \rightarrow \Xc$. Note that for any $\Fb_q$-scheme $T$ and any $y,z \in \Xc(T)$ such that $y \sim x_T$ and $z \sim x_T$ in $\Xc(T)$, then $y,z$ factors through $x$ and $\underline{\mathrm{Isom}}_{\Xc}(y,z) \cong \underline{\mathrm{Isom}}_{[\Spec (\Fb_q)/\underline{\Aut}_x]}(y',z')$, where $y=\overline{x} \circ y'$ and $z=\overline{x} \circ z'$. Thus, $[\Spec (\Fb_q)/\underline{\Aut}_x]$ is a substack of $\Xc$ via $\overline{x}$.
        
        Now consider $\Ic(\Xc)_{\overline{x}}$ defined by the following Cartesian square:
        \[
            \begin{tikzcd}
                \Ic(\Xc)_{\overline{x}} \arrow[d] \arrow[r] & \Ic(\Xc) \arrow[d]\\
                \left[\Spec (\Fb_q)/\underline{\Aut}_x \right] \arrow[r,"\overline{x}"] & \Xc
            \end{tikzcd}
        \]
        This is a substack of $\Ic(\Xc)$, and $(y,\beta) \in (\Ic(\Xc)_{\overline{x}})(\Fb_q)$ iff $y \sim x$ in $\Xc(\Fb_q)$. Since $x$ contributes 1 on the unweighted point count $|\Xc(\Fb_q)/\sim|$, it suffices to show that $\#_q(\Ic(\Xc)_{\overline{x}})=1$.
        
        Observe that two points $(x,\alpha)$ and $(x,\beta)$ in $\Ic(\Xc)_{\overline{x}}$ are equivalent iff $\beta = \phi \circ \alpha \circ \phi^{-1}$ for some $\phi \in \Aut(x)$. This holds in general if we replace $x$ by $y: U \rightarrow \Xc$ that factors thru $x$. Thus, $\Ic(\Xc)_{\overline{x}} \cong [\underline{\Aut}_x/\underline{\Aut}_x]$, where the group space action is the conjugation. Since the diagonal of $\Xc$ is quasi-separated and of finite type, $\underline{\Aut}_x$ is a quasi-separated group algebraic space of finite type over $\Fb_q$ by Diagram~\eqref{eq:aut_fibprod}; henceforth, $\Aut(x)=\underline{\Aut}_x(\Fb_q)$ is a finite group since $\underline{\Aut}_x$ admits a finite stratification by schemes of finite type by \cite[II.6.6]{Knutson}. Moreover, $\Aut(x,\alpha)$ is the finite centralizer subgroup $C_{\Aut(x)}(\alpha) \subset \Aut(x)$, and the set $(\Ic(\Xc)_{\overline{x}})(\Fb_q)/\sim$ is exactly the set $Cl(\Aut(x))$ of orbits of $\Aut(x)$ under the conjugation. Then, the Orbit-Stabilizer Theorem implies that as a set,
        \[
            \Aut(x)\cong \bigsqcup_{\alpha \in Cl(\Aut(x))} \Aut(x)/C_{\Aut(x)}(\alpha) .
        \]
        Finally, we can divide the cardinality of both sides by the finite number $|\Aut(x)|$; then right hand side becomes $\#_q(\Ic(\Xc)_{\overline{x}})$, proving the statement.
    \end{proof}

    \par The following Lemma shows that certain nice property of $\Xc$ carries over to $\Ic(\Xc)$ as well.
    
    \begin{lem}\label{lem:affdiag}
        If $\Xc$ is an algebraic stack of finite type with affine finite type diagonal, then so is $\Ic(\Xc)$.
    \end{lem}
    
    \begin{proof}
        Since $\Xc$ is of finite type with finite type diagonal, $\Ic(\Xc)$ must be of finite type as well by Diagram~\eqref{eq:inertia_fibprod}. It remains to show that $\Ic(\Xc)$ has an affine diagonal. This is equivalent to showing that for any scheme $T$ and any pairs $(x,\alpha),(y,\beta) \in \Ic(\Xc)(T)$, the Isom space $\underline{\mathrm{Isom}}_{\Ic(\Xc)}((x,\alpha),(y,\beta))$ is an affine $T$-scheme by the following Cartesian diagram:
        \[
            \begin{tikzcd}
                \underline{\mathrm{Isom}}_{\Ic(\Xc)}((x,\alpha),(y,\beta)) \arrow[d] \arrow[r] & T \arrow[d,"{(x,\alpha)} \times {(y,\beta)}"]\\
                \Ic(\Xc) \arrow[r,"\Delta"] & \Ic(\Xc) \times \Ic(\Xc)
            \end{tikzcd}
        \]
        
        To see the structure of $\underline{\mathrm{Isom}}_{\Ic(\Xc)}((x,\alpha),(y,\beta)) \rightarrow T$, observe that $\underline{\mathrm{Isom}}_{\Xc}(x,y) \rightarrow T$ and $\underline{\Aut}_x \rightarrow T$ are affine morphisms of finite type by the conditions on the diagonal of $\Xc$. Then $\underline{\mathrm{Isom}}_{\Ic(\Xc)}((x,\alpha),(y,\beta))$ is the preimage under the closed subscheme $1_x \in \underline{\Aut}_x$ of a morphism between affine $T$-schemes:
        \begin{align*}
            \underline{\mathrm{Isom}}_{\Xc}(x,y) & \rightarrow  \underline{\Aut}_x\\
            \phi & \mapsto  \phi \circ \alpha \circ \phi^{-1} \circ \beta^{-1}
        \end{align*}
        Therefore, $\underline{\mathrm{Isom}}_{\Ic(\Xc)}((x,\alpha),(y,\beta))$ is an affine $T$-scheme as well.
    \end{proof}
    
    In practice, an algebraic stack $\Xc$ can be characterized by its smooth cover $U \rightarrow \Xc$ by an algebraic space $U$ (most of the time, $U$ is assumed to be a scheme) with the space of equivalence relations $R$, i.e., $R$ is defined via the following Cartesian diagram
    \[
        \begin{tikzcd}
            R \arrow[d,"s"] \arrow[r,"t"] & U \arrow[d]\\
            U  \arrow[r] & \Xc
        \end{tikzcd}
    \]
    where $s(r)=x$ and $t(r)=y$ for any directed equivalence relation $r:x \rightarrow y$ of $x,y \in U$ via $r \in R$. In this case, $\Xc \cong [U/R]$ (technically, $[U/R]$ is the sheafification of the presheaf $U/R$ of groupoids over $\mathrm{Sch}$). Note that the following Cartesian square
    \[
        \begin{tikzcd}
            R \arrow[d] \arrow[r,"s \times t"] & U \times U \arrow[d]\\
            \Xc  \arrow[r,"\Delta"] & \Xc \times \Xc
        \end{tikzcd}
    \]
    implies that $R$ is an algebraic space as well, since $\Delta: \Xc \rightarrow \Xc \times \Xc$ is representable. Given this presentation, we obtain the following presentation of $\Ic(\Xc)$:
    
    \begin{prop}\label{prop:presentation_inertia}
        $\Ic(\Xc) \cong [R_{\Delta}/(R_{\Delta} \tensor{\times}{_s} R)]$, where $R_{\Delta}$ is defined by the following Cartesian square:
        \[
            \begin{tikzcd}
                R_{\Delta} \arrow[d] \arrow[r] & R \arrow[d]\\
                U  \arrow[r,"\Delta"] & U \times U
            \end{tikzcd}
        \]
    \end{prop}
    \begin{proof}
        By \cite[Tag 06PR]{Stacks}, $R_{\Delta}$ (denoted $G$ in loc.cit.) is a smooth cover of $\Ic(\Xc)$. To see that $R_{\Delta} \times_{\Ic(\Xc)} R_{\Delta}$ is isomorphic to $R_{\Delta} \tensor[_t]{\times}{_s} R$, it suffices to compare their $T$-points for any scheme $T$. Recall by the Cartesian diagram above that any $T$-point of $R_{\Delta}$ is characterized by a pair $(u,r) \in U \times R$ where $r: u \rightarrow u$. Then, given any $((u_1,r_1),(u_2,r_2)) \in (R_{\Delta} \times_{\Ic(\Xc)} R_{\Delta})(T)$, then there is $\tau: u_1 \rightarrow u_2$ in $R(T)$ such that $\tau \circ r_1=r_2 \circ \tau$. This gives an element $((u_1,r_1),\tau) \in (R_{\Delta} \tensor[_t]{\times}{_s} R)(T)$. The converse can be recovered, as $u_2$ is the target of $r_1$ and $r_2=\tau \circ r_1 \circ \tau^{-1}$. This establishes the bijection between $T$-points of $R_{\Delta} \times_{\Ic(\Xc)} R_{\Delta}$ and $R_{\Delta} \tensor[_t]{\times}{_s} R$.
    \end{proof}
    
    \par Sometimes, we will denote $R(\Xc)$ (resp, $R_{\Delta}(\Xc)$) instead of $R$ (resp, $R_{\Delta}$) when we need to emphasize the algebraic stack $\Xc$ in question.

    \medskip
    
    \par Recall that a quotient stack, denoted $[U/G]$, corresponds to $U$ a scheme with the action of a group scheme $G$. In this case, $R=U \times G$ with $s$ being the first projection and $t$ being the $G$-action map $t: (u,g) \mapsto g\cdot u$. Then $R_{\Delta} \subset U \times G$ consists of $(u,g)$ with $t(u,g)=g \cdot u=u$. 
    
    \begin{cor}\label{cor:presentation_inertia}
        If $\Xc \cong [U/G]$ is a quotient stack, then $\Ic(\Xc)$ is also a quotient stack $[R_{\Delta}/G]$, where $R \cong U \times G$ and $G$ acts on $R_{\Delta} \subset R$ by $g \cdot (u,h)=(g \cdot u, ghg^{-1})$.
    \end{cor}
    \begin{proof}
        By the proof of Proposition~\ref{prop:presentation_inertia}, it suffices to show that $R_{\Delta} \tensor[_t]{\times}{_s} R \cong R_{\Delta} \times G$ and that the action map 
        \begin{align*}
        R_{\Delta} \tensor[_t]{\times}{_s} R &\rightarrow R_{\Delta}\\
        ((u,r),\tau) &\mapsto (t(r),\tau \circ r \circ \tau^{-1}),
        \end{align*}
        which coincides with the second projection of $R_{\Delta} \times_{\Ic(\Xc)} R_{\Delta}$, coincides with the conjugate $G$-action described above. The isomorphism $R_{\Delta} \tensor[_t]{\times}{_s} R \cong R_{\Delta} \times G$ is given by:
        \begin{align*}
            R_{\Delta} \times G &\rightarrow R_{\Delta} \tensor[_t]{\times}{_s} R\\
            ((u,g),h) & \mapsto ((u,g),(g \cdot u=u, h))
        \end{align*}
        By the description of the action map above, $G$ acts on $R_{\Delta}$ by the conjugation.
    \end{proof}
    
    \par Now assume that a quotient stack $\Xc \cong [U/G]$ of finite type has the affine diagonal. Then, $R_{\Delta}$ is not irreducible in general; in fact, not even connected. Since the image of the second projection $\pi_2: U \times G \supset R_{\Delta} \rightarrow G$ can have many irreducible components $G_i$, we have the decomposition $R_{\Delta}=\cup \pi_2^{-1}(G \cdot G_i)$ (where $G$ acts on itself by conjugation). Note that when $\pi_2(R_{\Delta})$ is disconnected, so is $R_{\Delta}$. 
    
    \medskip
    
    \par Thus, assume furthermore that $\Xc$ is a Deligne-Mumford (DM) stack. Since the diagonal of $\Xc$ is affine (by the previous assumption) and formally unramified (by DM), the diagonal must be finite; this implies that $\pi_2(R_{\Delta})$ lies in torsion subset of $G$. Instead of stratifying $\pi_2(R_{\Delta})$ by $G$-orbits of its irreducible components as above, Abramovich-Graber-Vistoli in \cite[Definition 3.1.5]{AGV} stratify $\Ic(\Xc)$ by looking at orders of automorphism elements: in our language, this induces a coarser stratification of $R_{\Delta}$:
    \begin{equation}\label{eq:decomp_order}
    \begin{split}
        R_{\Delta}(\Xc)&=\bigsqcup_{r \in \Zb_{>0}} R_{\Delta, \mu_r}(\Xc)\\
        \Ic(\Xc)&= \bigsqcup_{r \in \Zb_{>0}} \Ic_{\mu_r}(\Xc)=\bigsqcup_{r \in \Zb_{>0}} [R_{\Delta,\mu_r}(\Xc)/G]
    \end{split}
    \end{equation}
    where $R_{\Delta,\mu_r}(\Xc)$ is the preimage under $\pi_2$ of the subscheme of order $r$ elements of $G$. However, $R_{\Delta,\mu_r}(\Xc)$ can still be disconnected with many components of different dimensions.

    \medskip
    
    \par Instead, assume that we have chosen a nice presentation of $\Xc$ into a quotient stack $[U/G]$ such that the support of $\pi_2(R_{\Delta})$ consists of finitely many closed points of $G$. In this case, $\pi_2(R_{\Delta})$ is, as a set, a disjoint union of conjugate classes of some closed points in $\pi_2(R_{\Delta})$. Let's use our initial decomposition of $R_{\Delta}$ as above by $G$-orbits of connected components of $\pi_2(R_{\Delta})$. This induces the following stratification:
    
    \begin{defn}\label{def:decomp_conjugate}
    Let $\Xc \cong [U/G]$ be a Deligne-Mumford quotient stack of finite type with affine diagonal and let $R_{\Delta}$ be as in Corollary~\ref{cor:presentation_inertia} such that the support of the second projection $\pi_2(R_{\Delta})$ in $G$ consists of finitely many closed points of $G$. Then the decomposition of the inertia stack $\Ic(\Xc)$ via the conjugacy classes is as follows:
    \begin{equation*}
    \begin{split}
        R_{\Delta}(\Xc)&=\bigsqcup_{\alpha \in Cl(G)} R_{\Delta, \alpha}(\Xc),\\
        \Ic(\Xc)&= \bigsqcup_{\alpha \in Cl(G)} \Ic_{\alpha}(\Xc)=\bigsqcup_{\alpha \in Cl(G)} [R_{\Delta,\alpha}(\Xc)/G],
    \end{split}
    \end{equation*}
    where $R_{\Delta,\alpha}(\Xc)$ is the preimage under $\pi_2$ of a conjugate class $\alpha \in Cl(G)$, as a finite subset of $G$.
    \end{defn}
    
    Note that $R_{\Delta,\alpha} = \sqcup_{g \in \alpha} R_{\Delta,g}$ where $R_{\Delta,g}$ is the preimage under $\pi_2$ of $g \in G$; it is the base change by $\kappa(g)/K$ of the fixed locus in $U$ of $g \in G$ (i.e., every point is fixed under the action of $g$). Observe that $R_{\Delta,hgh^{-1}}=h \cdot R_{\Delta,g}$, which is itself $R_{\Delta,g}$ (then $h \in C_G(g)$) or is disjoint from $R_{\Delta,g}$ by the finiteness of $\pi_2(R_{\Delta})$. Therefore, $\Ic_{\alpha}(\Xc)=[R_{\Delta,g}/C_G(g)]$ for any generator $g \in \alpha$, i.e. $\alpha=G \cdot g$. 

    \medskip
    
    As a summary, the decomposition in Definition~\ref{def:decomp_conjugate} is finer than \eqref{eq:decomp_order} when it exists, but assumes \textit{the finiteness of $\pi_2(R_{\Delta}) \subset G$ as a subset}. We will see that weighted projective stacks (and Hom stacks) defined in \S 3 satisfy this condition.
    
    \begin{rmk}\label{rmk:finiteness_pi2_into_G}
        When $\Xc \cong [U/G],~U,~G$ in Definition~\ref{def:decomp_conjugate} are defined over a perfect field $K$, the condition, that the support of $\pi_2(R_{\Delta})$ in $G$ consists of finitely many closed points of $G$, is equivalent to the finiteness of the following set:
        \[
            \left\{ g \text{ is a geometric point of } G \;|\; g\cdot u=u \text{ for some geometric point } u \text{ of } U \right\} .
        \]
        When $\Xc$ is Deligne-Mumford and $G$ is an abelian group (such as $\Gb_m$), this is easy to check. However, when $G$ is a non-abelian group (examples are GIT constructions of moduli of smooth/stable curves), this condition puts restriction on what kind of $g$ can fix an element of $U$, even when $\Xc$ is a Deligne-Mumford stack. If $g \cdot u=u$, then $hgh^{-1} \cdot hu=hu$, so that this set above is a union of conjugacy classes as sets. Whenever the centralizer subgroup scheme $C_G(g)$ has lower dimension than $G$, the conjugacy class (i.e., the orbit of $g$ under conjugation) forms a positive dimensional subscheme, contained in the set above. Since $K$ is perfect, the algebraic closure $\overline{K}$ is infinite, implying that such positive dimensional subschems have infinitely many geometric points by Bertini's Theorem. 
    \end{rmk}


    \section{Hom stack $\Hom_n(\Pb^1,\Pcv)$ of rational curves on a weighted projective stack}
    \label{sec:ModuliofHom}

    \par In this section, we formulate the Hom stack $\Hom_n(\Pb^1,\Pcv)$ over a base field $K$. First, we recall the definition of a weighted projective stack $\Pcv$ with the weight $\vec{\lambda}$ over $K$.
    
    \begin{defn}\label{def:wtproj} 
        Fix a tuple of nondecreasing positive integers $\vec{\lambda} = (\lambda_0, \dotsc, \lambda_N)$. The $N$-dimensional weighted projective stack $\Pcv = \Pc(\lambda_0, \dotsc, \lambda_N)$ with the weight $\vec{\lambda}$ is defined as a quotient stack
        \[
        \Pcv := \left[(\Ab_{x_0, \dotsc, x_N}^{N+1} \setminus 0) / \Gb_m\right]
        \]
        where $\zeta \in \Gb_m$ acts by $\zeta \cdot (x_0, \dotsc, x_N)=(\zeta^{\lambda_0} x_0, \dotsc, \zeta^{\lambda_N} x_N)$. In this case, the degree of $x_i$'s are $\lambda_i$'s respectively. A line bundle $\Oc_{\Pcv}(m)$ is defined to be the line bundle associated with the sheaf of degree $m$ homogeneous rational functions without poles on $\Ab_{x_0, \dotsc, x_N}^{N+1} \setminus 0$.
    \end{defn}
    
    \par Note that $\Pcv$ is not an (effective) orbifold when $\gcd(\lambda_0, \dotsc, \lambda_N) \neq 1$. In this case, the finite cyclic group scheme $\mu_{\gcd(\lambda_0, \dotsc, \lambda_N)}$ is the generic stabilizer of $\Pcv$. When we need to emphasize the field $K$ of definition of $\Pcv$, we instead use the notation $\Pc_{K}(\vec\lambda)$.

    \begin{lem}\label{lem:wtproj_prop}
        The $N$-dimensional weighted projective stack $\Pcv = \Pc(\lambda_0, \dotsc, \lambda_N)$ over any field $K$ is of finite type with finite type affine diagonal.
    \end{lem}
    
    \begin{proof}
        Since the smooth schematic cover $\Ab_{x_0, \dotsc, x_N}^{N+1} \setminus 0$ of $\Pcv$ is of finite type over $K$, $\Pcv$ is of finite type over $K$ as well. It remains to prove the properties of the diagonal of $\Pcv$. Choose any $T$-point $x=(x_0,\dotsc,x_N)$ of $U:=\Ab^{N+1}_K \setminus 0$. The fiber over $x$ of $R_{\Delta} \rightarrow U$ as in Corollary~\ref{cor:presentation_inertia} is a proper subgroup scheme of $\Gb_m$ (over $T$), which is always affine of finite type over $T$. Henceforth, the diagonal of $\Pcv$ satisfies the desired properties.
    \end{proof}
    
    However, when $K=\Fb_p$ for some prime $p$, $\Pc(1,p)$ is not Deligne-Mumford, as $\underline{\Aut}_{[0:1]} \cong \mu_p$, which is not formally unramified over $\Fb_p$. Nevertheless, the following proposition shows that any $\Pcv$ behaves well in most characteristics as a tame Deligne--Mumford stack:
    
    \begin{prop}\label{prop:wtprojtame}
        The weighted projective stack $\Pcv = \Pc(\lambda_0, \dotsc, \lambda_N)$ is a tame Deligne--Mumford stack over $K$ if $\mathrm{char}(K)$ does not divide $\lambda_i \in \N$ for every $i$.
    \end{prop}
    
    \begin{proof}
        For any algebraically closed field extension $\overline{K}$ of $K$, any point $y \in \Pcv(\overline{K})$ is represented by the coordinates $(y_0,\dotsc,y_N) \in \Ab_{\overline{K}}^{N+1}$ with its stabilizer group as the subgroup of $\Gb_m$ fixing $(y_0,\dotsc,y_N)$. Hence, any stabilizer group of such $\overline{K}$-points is $\Zb/u\Zb$ where $u$ divides $\lambda_i$ for some $i$. Since the characteristic of $K$ does not divide the orders of $\Zb/\lambda_i\Zb$ for any $i$, the stabilizer group of $y$ is $\overline{K}$-linearly reductive. Hence, $\Pcv$ is tame by \cite[Theorem 3.2]{AOV}. Note that the stabilizer groups constitute fibers of the diagonal $\Delta: \Pcv \rightarrow \Pcv \times_K \Pcv$. Since $\Pcv$ is of finite type and $\Zb/u\Zb$'s are unramified over $K$ whenever $u$ does not divide $\lambda_i$ for some $i$, $\Delta$ is unramified as well. Therefore, $\Pcv$ is also Deligne--Mumford by \cite[Theorem 8.3.3]{Olsson2}.
    \end{proof}
    
    \par The tameness is analogous to flatness for stacks in positive/mixed characteristic as it is preserved under base change by \cite[Corollary 3.4]{AOV}. Moreover, if a stack $\mathcal X$ is tame and Deligne--Mumford, then the formation of the coarse moduli space $c: \mathcal X \rightarrow X$ commutes with base change as well by \cite[Corollary 3.3]{AOV}.
    
    \begin{exmp}
        When the characteristic of the field $K$ is not equal to 2 or 3, \cite[Proposition 3.6]{Hassett2} shows that one example is given by the proper Deligne--Mumford stack of stable elliptic curves $(\Me)_K \cong [ (\Spec~K[a_4,a_6]-(0,0)) / \Gb_m ] = \Pc_K(4,6)$ by using the short Weierstrass equation $y^2 = x^3 + a_4x + a_6x$, where $\zeta \cdot a_i=\zeta^i \cdot a_i$ for $\zeta \in \Gb_m$ and $i=4,6$. Thus, $a_i$'s have degree $i$'s respectively. Note that this is no longer true if characteristic of $K$ is 2 or 3, as the Weierstrass equations are more complicated. 
    \end{exmp}
    
    In the proof of Lemma~\ref{lem:wtproj_prop}, we have shown that $R_{\Delta} \rightarrow U$ is proper, implying that $\pi_2(R_{\Delta}) \subset \Gb_m$ is a proper subgroup scheme, i.e., supported on finitely many closed points. Thus, we can apply the decomposition in Definition~\ref{def:decomp_conjugate} to the inertia stack $\Ic(\Pcv)$:
    \begin{prop}\label{prop:wtproj_inertia_decomp}
        For any $N$-dimensional weighted projective stack $\Pc_K(\vec\lambda)$, Definition~\ref{def:decomp_conjugate} describes connected components of $\Ic(\Pc_K(\vec\lambda))$:
        \[
            \Ic(\Pc_K(\vec\lambda)) \cong \bigsqcup_{g \in |(\Gb_m)_K|} \Pc_{\kappa(g)}(\vec\lambda_{I_g})
        \]
        where $|(\Gb_m)_K|$ is set of closed points of $(\Gb_m)_K$, $I_g$ is the largest subset of $\{0,\dotsc,N\}$ such that $\mathrm{ord}(g)$ divides $\gcd_{i \in I_g}(\lambda_i)$, and $\vec\lambda_{I_g}$ is the subtuple of $\vec{\lambda}$ indexed by $I_g \subset \{0,\dotsc,N\}$.
    \end{prop}
    Note that $I_{g}=I_{g'}$ when $\mathrm{ord}(g)=\mathrm{ord}(g')$, as any subgroup of $\Gb_m$ is cyclic. Also, when $|I_g|=0$, then $\Pc(\vec\lambda_{I_g}) = \emptyset$ vacuously.
    \begin{proof}[Proof of Proposition~\ref{prop:wtproj_inertia_decomp}]
        It suffices to show that $R_{\Delta,g}$ is the subspace 
        \[
        \{(x_0,\dotsc,x_N,g) \in (\Ab^{N+1}\setminus 0) \times \Gb_m \; | \; x_i = 0 \; \mathrm{if} \; \mathrm{ord}(g) \text{ does not divide } \lambda_i\}
        \]
        as commutativity of $\Gb_m$ implies that $C_G(g)=\Gb_m$ for any $g \in \Gb_m$ (here, $g$ as a closed point of $\Gb_m$ in above coordinates is equivalent to taking a Galois orbit of a representative of $g$ as a $\kappa(g)$-point of $\Gb_m$). Note that this space is a $\kappa(g)$-variety as its projection onto $\Gb_m$ maps to $g \in \Gb_m$. For any $g \in |(\Gb_m)_K|$, $x=(x_0,\dotsc,x_N) \in \Ab^{N+1}$ is fixed by $g$ iff $g^{\lambda_i}x_i=x_i$ for all $i$. Whenever $g^{\lambda_i} \neq 1$, $x_i$ must be zero. Thus, $x$ lies in the closed subscheme $\{x_i=0 \; : \; \forall i, \; g^{\lambda_i} \neq 1\}$, which is exactly the desired subspace.
    \end{proof}
    
    \par We now generalize the Hom stack formulation to $\Pcv$ as follows:

    \begin{prop}\label{prop:DMstack}
    Hom stack $\Hom_n(\Pb^1,\Pcv)$ with weight $\vec{\lambda} = (\lambda_0, \dotsc, \lambda_N)$, which parameterize degree $n \in \Nb$ $K$-morphisms $f:\Pb^1 \rightarrow \Pcv$ with $f^*\Oc_{\Pcv}(1) \cong \Oc_{\Pb^1}(n)$ over a base field $K$ with $\mathrm{char}(K)$ not dividing $\lambda_i \in \N$ for every $i$, is a smooth separated tame Deligne--Mumford stack of finite type with $\mathrm{dim_K}\left(\Hom_n(\Pb^1,\Pcv)\right) = |\vec{\lambda}|n + N$ where $|\vec{\lambda}|:=\sum\limits_{i=0}^{N} \lambda_i$ .
    \end{prop}

    \begin{proof} 
    $\Hom_n(\Pb^1,\Pcv)$ is a smooth Deligne--Mumford stack by \cite[Theorem 1.1]{Olsson}. It is isomorphic to the quotient stack $[T/\Gb_m]$, admitting a smooth schematic cover $T \subset \left( \bigoplus\limits_{i=0}^{N} H^0(\Oc_{\Pb^1}(\lambda_i \cdot n))\right) \setminus 0$, parameterizing the set of tuples $(u_0, \dotsc, u_N)$ of sections with no common zero (here, we interpret $H^0(\Oc_{\Pb^1}(\lambda_i \cdot n))$ as an affine space over $K$ of appropriate dimension, induced by its $K$-vector space structure). The $\Gb_m$ action on $T$ is given by $\zeta \cdot (u_0, \dotsc, u_N)=(\zeta^{\lambda_0}u_0,\dotsc,\zeta^{\lambda_N}u_N)$ . Note that 
    \[\dim T=\sum_{i=0}^N h^0(\Oc_{\Pb^1}(\lambda_i \cdot n))=\sum_{i=0}^N (\lambda_i \cdot n+1)=|\vec{\lambda}|n+N+1,\]
    implying that $\dim \Hom_n(\Pb^1,\Pcv)=|\vec{\lambda}|n+N$ since $\dim \Gb_m=1$.
    
    As $\Gb_m$ acts on $T$ properly with positive weights $\lambda_i \in \N$ for every $i$, the quotient stack $[T/\Gb_m]$ is separated. It is tame as in \cite[Theorem 3.2]{AOV} since $\mathrm{char}(K)$ does not divide $\lambda_i$ for every $i$ .
    \end{proof}
    
    \begin{rmk}\label{rmk:Hom_substack_wtproj}
        In the proof of Proposition~\ref{prop:DMstack}, we showed that $\Hom_n(\Pb^1,\Pcv) \cong [T/\Gb_m]$ where $T$ is an open dense $\Gb_m$-invariant subscheme of $\oplus H^0(\Oc_{\Pb^1}(\lambda_i\cdot n))$ not containing zero, where for each $i$, $\Gb_m$ acts on $H^0(\Oc_{\Pb^1}(\lambda_i\cdot n))$ with weight $\lambda_i$. In fact, this remains true even when the characteristic assumption fails, as the arguments still follow. Since $h^0(\Oc_{\Pb^1}(\lambda_i\cdot n))=n\lambda_i+1$, $\Hom_n(\Pb^1,\Pcv)$ is an open substack of $\Pov$ where 
        \[
            \Lambdavec := (\underbrace{\lambda_0, \ldots, \lambda_0}_{n\lambda_0+1 \text{ times}}, \ldots, \underbrace{\lambda_N, \ldots, \lambda_N}_{n\lambda_N+1 \text{ times}} ) \,.
        \]
        Furthermore, $(u_0,\dotsc,u_N) \in \oplus H^0(\Oc_{\Pb^1}(\lambda_i\cdot n))$ lies in $T$ iff $u_i$'s have no common zero on $\Pb^1$. By Lemma~\ref{lem:wtproj_prop}, $\Hom_n(\Pb^1,\Pcv)$ is of finite type with finite type affine diagonal (without any condition on the base field $K$) $\qed$
    \end{rmk}

    \par Similar to $\Pcv$, the inertia stack $\Ic(\Hom_n(\Pb^1,\Pcv))$ also admits a clear decomposition (i.e., each summand is the Hom stack $\Hom_n(\Pb^1,\Pc(\vec\lambda_{I_g}))$) that will play a crucial role.

    \begin{prop}\label{prop:Hom_inertia_decomp}
        The inertia stack of the Hom stack $\Hom_n(\Pb^1,\Pcv)$ admits the following decomposition into connected components as in Definition~\ref{def:decomp_conjugate}, where $I_g$ and $\vec\lambda_{I_g}$ are the same as in Proposition~\ref{prop:wtproj_inertia_decomp}:
        \[
            \Ic(\Hom_n(\Pb_K^1,\Pc_K(\vec\lambda))) \cong \bigsqcup_{g \in |(\Gb_m)_K|} \Hom_n(\Pb_{\kappa(g)}^1,\Pc_{\kappa(g)}(\vec\lambda_{I_g}))
        \]
    \end{prop}
    Note that $\Hom_n(\Pb^1,\Pc(\vec\lambda_{I_g})) = \emptyset$ whenever $|I_g| \le 1$, as there are no maps from $\Pb^1$ to $\Pc(\vec\lambda_{I_g})$ where the pullback of $\Oc(1)$ to $\Pb^1$ has degree $n$.
    \begin{proof}[Proof of Proposition~\ref{prop:Hom_inertia_decomp}]
        By Remark~\ref{rmk:Hom_substack_wtproj}, $\Hom_n(\Pb^1,\Pcv)$ is an open substack of $\Pov$. Restricting the decomposition of $\Ic(\Pov)$ as in Proposition~\ref{prop:wtproj_inertia_decomp} to $\Hom_n(\Pb^1,\Pcv)$,
        \[
            \Ic(\Hom_n(\Pb^1,\Pcv)) \cong \bigsqcup_{g \in |(\Gb_m)_K|} \left[T \cap \{\vec{u} \in \oplus H^0(\Oc_{\Pb^1}(\lambda_i\cdot n)) \; : \; u_i=0 \text{ if } g^{\lambda_i} \neq 1\}/\Gb_m \right]_{\kappa(g)}
        \]
        Since $\vec{u} \in T$ iff $u_i$'s have no common zeroes, each summand is isomorphic to $\Hom_n(\Pb^1,\Pc(\vec\lambda_{I_g}))$.
    \end{proof}


    \section{Motive/Point count of $\mathrm{Hom}$ and inertia stacks}
    \label{sec:count}

    \par In this section, we use the idea of cut-and-paste by Grothendieck and acquire the motive $\{\Xc\} \in K_0(\mathrm{Stck}_\mathrm{K})$ of moduli stack $\Xc$ in the Grothendieck ring of $K$--stacks; in fact, we show that $\{\Xc\}$ is a polynomial in the Lefschetz motive $\Lb:=\{\Ab^1_K\}$. Particularly, we acquire the motive $\{\mathrm{Poly}_1^{(d_1,\dotsc,d_m)}\}$ of the space of monic coprime polynomials through filtration which in turn provides the motive $\left\{\Hom_n(\Pb^1,\Pcv)\right\}$ of the Hom stack through stratification. In the end, we acquire the weighted point count of the inertia stack $\left\{\Ic\left(\Hom_n(\Pb^1,\Pcv)\right)\right\}$ of the Hom stack over $\Fb_q$ through the decomposition of Proposition~\ref{prop:Hom_inertia_decomp}.

    \medskip

    \par First, we recall the definition of the Grothendieck ring of algebraic stacks following \cite{Ekedahl}.

    \begin{defn}\label{defn:GrothringStck}
        \cite[\S 1]{Ekedahl}
        Fix a field $K$. Then the \emph{Grothendieck ring $K_0(\mathrm{Stck}_K)$ of algebraic stacks of finite type over $K$ all of whose stabilizer group schemes are affine} is an abelian group generated by isomorphism classes of $K$-stacks $\{\Xc\}$ of finite type, modulo relations:
        \begin{itemize}
            \item $\{\Xc\}=\{\Zc\}+\{\Xc \setminus \Zc\}$ for $\Zc \subset \Xc$ a closed substack,
            \item $\{\Ec\}=\{\Xc \times \Ab^n \}$ for $\Ec$ a vector bundle of rank $n$ on $\Xc$.
        \end{itemize}
        Multiplication on $K_0(\mathrm{Stck}_K)$ is induced by $\{\Xc\}\{\Yc\}:=\{\Xc \times_K \Yc\}$. There is a distinguished element $\Lb:=\{\A^1\} \in K_0(\mathrm{Stck}_K)$, called the \emph{Lefschetz motive}.
    \end{defn}
    
    \par Given an algebraic $K$-stack $\Xc$ of finite type with affine diagonal, the \emph{motive} of $\Xc$ refers to $\{\Xc\} \in K_0(\text{Stck}_K)$.

    \medskip

    \par As the Grothendieck ring $K_0(\mathrm{Stck}_K)$ is the universal object for \textit{additive invariants}, it is easy to see that when $K=\Fb_q$, the assignment $\{X\} \mapsto \#_q(X)$ gives a well-defined ring homomorphism $\#_q: K_0(\mathrm{Stck}_{\Fb_q}) \rightarrow \Q$ (c.f. \cite[\S 2]{Ekedahl}) rendering the weighted point count of a stack $\Xc$ over $\Fb_q$. Note that $\#_q(\Xc) < \infty$ when $\Xc$ is of finite type (see discussion right below Definition~\ref{def:wtcount}).

    \medskip
    
    \par Since many algebraic stacks can be written locally as a quotient of a scheme by an algebraic group $\Gb_m$, the following lemma (a special case of \cite[\S 1]{Ekedahl}) is very useful:
    
    \begin{lem}\label{lem:Gm_quot}
        \cite[Lemma 15]{HP} For any $\Gb_m$-torsor $\mathcal X \rightarrow \mathcal Y$ of finite type algebraic stacks, we have $\{\mathcal Y\}=\{\mathcal X\}\{\Gb_m\}^{-1}$.
    \end{lem}

    \par The subsequent proofs involves the following variety of its own interest (a slight generalization of \cite[Definition 1.1]{FW}) :
    
    \begin{defn}\label{def:poly}
        Fix $m \in \Zb_{>0}$ and $d_1,\dotsc,d_m \ge 0$. Define $\mathrm{Poly}_1^{(d_1,\dotsc,d_m)}$ as the set of tuples $(f_1,\dotsc,f_m)$ of monic polynomials in $K[z]$ so that
        \begin{enumerate}
            \item $\deg f_i=d_i$ for each $i$, and
            \item $f_1,\dotsc,f_m$ have no common roots in $\overline{K}$.
        \end{enumerate}
    \end{defn}

    \par Since the set $\mathrm{Poly}_1^{(d_1,\dotsc,d_m)}$ is open inside the affine space (complement of the resultant hypersurface) parameterizing the tuples of monic coprime polynomials of degrees $(d_1,\dotsc,d_m)$, we can endow $\mathrm{Poly}_1^{(d_1,\dotsc,d_m)}$ with a structure of affine variety defined over $\Zb$.

    \medskip
    
    \par Generalizing the proof of \cite[Theorem 1.2]{FW} with the correction from \cite[Proposition 3.1.]{PS}, we find the motive of $\mathrm{Poly}_1^{(d_1,\dotsc,d_m)}$ :
    
    \begin{prop}[Motive of the Poly space $\mathrm{Poly}_1^{(d_1, \cdots, d_m)}$ over $K$]
        \label{prop:poly_m}
        Fix $0 \le d_1 \le d_2 \le \cdots \le d_m$. Then, 
        \[ \left\{\mathrm{Poly}_1^{(d_1, \cdots, d_m)}\right\}=
        \begin{cases}
        \Lb^{d_1 + \cdots + d_m}-\Lb^{d_1 + \cdots + d_m - m+1} & \text{ if } d_1 \neq 0\\
        \Lb^{d_1 + \cdots + d_m} & \text{ if } d_1=0
        \end{cases}
        \]
    \end{prop}
    
    \begin{proof}
        The proof is analogous to \cite[Theorem 1.2 (1)]{FW}, with the correction from \cite[Proposition 3.1.]{PS}, and is a direct generalization of \cite[Proposition 18]{HP}. Here, we recall the differences to the work in \cite{FW, HP, PS}. \\
        
        \noindent \textbf{Step 1}: The space of $(f_1,\dotsc,f_m)$ monic polynomials of degree $d_1,\dotsc,d_m$ is instead the quotient $\A^{d_1} \times \dotsb \times \A^{d_m}/(S_{d_1} \times \dotsm \times S_{d_m}) \cong \A^{d_1 + \dotsb + d_m}$. We have the same filtration of $\A^{\sum d_i}$ by $R_{1,k}^{(d_1,\dotsc,d_m)}$: the space of monic polynomials $(f_1,\dotsc,f_m)$ of degree $d_1,\dotsc,d_m$ respectively for which there exists a monic $h \in K[z]$ with $\deg (h) \ge k$ and monic polynomials $g_i \in K[z]$ so that $f_i=g_ih$ for any $i$. The rest of the arguments follow analogously, keeping in mind that the group action is via $S_{d_1} \times \dotsb \times S_{d_m}$. \\

        \noindent \textbf{Step 2}: Here, we prove that $\{R_{1,k}^{(d_1,\dotsc,d_m)}-R_{1,k+1}^{(d_1,\dotsc,d_m)}\} = \{\mathrm{Poly}_1^{(d_1-k,\dotsc,d_m-k)}\times \A^k\}$. Just as in \cite{FW}, the base case of $k=0$ follows from the definition (in fact, loc.cit. shows that the two schemes are indeed isomorphic). For $k \ge 1$, \cite[Proposition 3.1]{PS} proves that the map
        \[\Psi: \mathrm{Poly}_1^{(d_1-k,\dotsc,d_m-k)} \times \Ab^k \rightarrow R_{1,k}^{(d_1,\dotsc,d_m)}\setminus R_{1,k+1}^{(d_1,\dotsc,d_m)}\]
        induces a piecewise isomorphism (where each piece is a locally closed subset, see \cite[Proposition 3.1]{PS} for more details); this immediately implies the claim by the definition of the Grothendieck Ring.\\
        
        \noindent \textbf{Step 3}: By combining Step 1 and 2 as in \cite{FW}, we obtain
        \[ \left\{\mathrm{Poly}_1^{(d_1,\dotsc,d_m)}\right\}=\Lb^{d_1+\dotsb+d_m}-\sum_{k \ge 1}\left\{\mathrm{Poly}_1^{(d_1-k,\dotsc,d_m-k)}\right\}\Lb^k \]

        \par For the induction on the class $\left\{\mathrm{Poly}_1^{(d_1,\dotsc,d_m)}\right\}$, we use lexicographic induction on the pair $(d_1,\dotsc,d_m)$. For the base case, consider when $d_1=0$. Here the monic polynomial of degree 0 is nowhere vanishing, so that any tuple of polynomials of degree $d_i$ for $i>1$ constitutes a member of $\mathrm{Poly}_1^{(0,d_2,\dotsc,d_m)}$, so that $\mathrm{Poly}_1^{(0,d_2,\dotsc,d_m)} \cong \mathbb{A}^{d_2+\dotsb+d_m}$. 
        
        \smallskip
        
        \par Now assume that $d_1>0$. Then, we obtain
        
        \begingroup
        \allowdisplaybreaks
        \begin{align*}
            &\left\{\mathrm{Poly}_1^{(d_1,\dotsc,d_m)}\right\}\\
            &=\Lb^{d_1+\dotsb+d_m}-\sum_{k \ge 1} \left\{\mathrm{Poly}_1^{(d_1-k,\dotsc,d_m-k)}\right\}\Lb^k\\
            &=\Lb^{d_1+\dotsb+d_m}-\left(\sum_{k=1}^{d_1-1}(\Lb^{(d_1-k)+\dotsb+(d_m-k)}-\Lb^{(d_1-k)+\dotsb+(d_m-k)-m+1})\Lb^k+\Lb^{(d_2-d_1)+\dotsb+(d_m-d_1)}\Lb^{d_1}\right)\\
            &=\Lb^{d_1+\dotsb+d_m}-\left(\sum_{k=1}^{d_1-1}(\Lb^{d_1+\dotsb+d_m-(m-1)k}-\Lb^{d_1+\dotsb+d_m-(m-1)(k+1)})+\Lb^{d_1+\dotsb+d_m-(m-1)d_1}\right)\\
            &=\Lb^{d_1+\dotsb+d_m}-\Lb^{d_1+\dotsb+d_m-m+1}
        \end{align*}
        \endgroup
    \end{proof}

    \subsection{Motive of Hom stack}\label{subsec:motive_Hom}
    
    Now we are ready to find the class in Grothendieck ring of the Hom stack $\Hom_n(\Pb^1,\Pcv)$:
    
    \begin{prop}
        \label{prop:motivecount_Hom} 
        Fix the weight $\vec{\lambda} = (\lambda_0, \dotsc, \lambda_N)$ with $|\vec{\lambda}|:=\sum\limits_{i=0}^{N} \lambda_i$. Then the motive of the Hom stack $\Hom_n(\Pb^1,\Pcv)$ in the Grothendieck ring of $K$--stacks $K_0(\mathrm{Stck}_{K})$ is equivalent to

        \[
        \begin{array}{ll}
        \left\{\Hom_n(\Pb^1,\Pcv)\right\} &= \left(\sum\limits_{i=0}^{N} \Lb^{i}\right) \cdot  \left(\Lb^{|\vec{\lambda}|n}-\Lb^{|\vec{\lambda}|n - N}\right) \\
        &\\
        &=\Lb^{|\vec{\lambda}|n - N} \cdot \left(\Lb^{2N} + \dotsb + \Lb^{N+1} - \Lb^{N-1} - \dotsb - 1\right)


        \end{array}
        \]
        where $\Lb^1:=\{\Ab^1_K\}$ is the Lefschetz motive.
    \end{prop}

    \begin{proof}
    Let $\vec{\lambda} = (\lambda_0, \dotsc, \lambda_N)$ and $\lambda_i \in \N$ for every $i$ with $|\vec{\lambda}|:=\sum\limits_{i=0}^{N} \lambda_i$. Then the Hom stack $\Hom_{n}(\Pb^{1}, \Pcv) \cong [T/\Gb_m]$ is the quotient stack by the proof of Proposition~\ref{prop:DMstack}. By Lemma~\ref{lem:Gm_quot}, we have $\{\Hom_{n}(\Pb^{1}, \Pcv)\}=(\Lb-1)^{-1}\{T\}$. Henceforth, it suffices to find the motive $\{T\}$, and not worry about the original $\Gb_m$-action on $T$. To do so, we need to reinterpret $T$ as follows.
    
    \par Fix a chart $\A^1 \hookrightarrow \Pb^1$ with $x \mapsto [1:x]$, and call $0=[1:0]$ and $\infty=[0:1]$. It comes from a homogeneous chart of $\Pb^1$ by $[Y:X]$ with $x:=X/Y$ away from $\infty$. Then for any $u \in H^0(\Oc_{\Pb^1}(d))$ with $d \ge 0$, $u$ is a homogeneous polynomial of degree $d$ in $X$ and $Y$. By substituting in $Y=1$, we obtain a representation of $u$ as a polynomial in $x$ with degree at most $d$. For instance, $\deg u < d$ as a polynomial in $x$ if and only if $u(X,Y)$ is divisible by $Y$ (i.e., $u$ vanishes at $\infty$). From now on, $\deg u$ means the degree of $u$ as a polynomial in $x$. Conventionally, set $\deg 0:=-\infty$.
    
    \par Therefore, $T$ parameterizes a $(N+1)$-tuple $(f_0,\dotsc,f_N)$ of polynomials in $K[x]$ with no common roots in $\overline{K}$, where $\deg f_i \le n\lambda_i$ for each $i$ with equality for some $i$. We use this interpretation to construct $\Phi: T \rightarrow \Ab^{N+1}\setminus 0$,
    $\Phi(f_0, \dotsc, f_N)= (a_0,\dotsc,a_N)$, where $a_i$ is the coefficient of degree $n\lambda_i$ term of $f_i$. 
    
    \par Now, we stratify $T$ by taking preimages under $\Phi$ of a stratification of $\Ab^{N+1}\setminus 0$ by $\sqcup E_J$, where $J$ is any proper subset of $\{0,\dotsc,N\}$ and
    \[ E_{J} = \{ (a_0, \dotsc , a_N)\; | \; a_j = 0 ~~\; \forall j \in J \} \iso \Gb_m^{N+1-|J|} \]
    Observe that $E_{J}$ has the natural free $\Gb_m^{N+1-|J|}$-action, which lifts to $\Phi^{-1}(E_{J})$ via multiplication on $\Gb_m$-scalars on $f_i$ for $i \notin J$. The action is free on $\Phi^{-1}(E_{J})$ as well, so that $\Phi|_{\Phi^{-1}(E_J)}$ is a Zariski-locally trivial fibration with base $E_J$. Each fiber is isomorphic to $F_{J}(n\vec{\lambda})$ defined below:
    \begin{defn}\label{def:Poly_modif}
    Fix $m \in \Nb$ and $\vec{d}:=(d_0,\dotsc,d_N) \in \Zb_{\ge 0}^{N+1}$. Given $J \subsetneq \{0,\dotsc,N\}$, $F_J(\vec{d})$ is defined as a variety consisting of tuples $(f_0,\dotsc,f_N)$ of $K$-polynomials without common roots such that
    \begin{itemize}
        \item for any $j \notin J$, then $f_j$ is monic of degree $n\lambda_j$, and
        \item for any $j \in J$, then $\deg f_j<n\lambda_j$ ($f_j$ is not necessarily monic).
    \end{itemize}
    If instead $J=\{0,\dotsc,N\}$, then we define $F_{J}(\vec{d}):=\emptyset$
    \end{defn}
    
    \par This implies that $\{\Phi^{-1}(E_J)\}=\{E_J\}\{F_J(n\vec{\lambda})\}=(\Lb-1)^{N+1-|J|}\{F_J(n\vec{\lambda})\}$. Since
    \begin{equation}\label{eq:sum_stratif}
        \{T\}=\sum_{J \subsetneq \{0,\dotsc,N\}} \{\Phi^{-1}(E_J)\}=\sum_{J \subsetneq \{0,\dotsc,N\}}\{E_J\}\{F_J(n\vec{\lambda})\} \; ,
    \end{equation}
    it suffices to find $\{F_J(n\vec{\lambda})\}$ as a polynomial of $\Lb$.
    
    \begin{prop}\label{prop:fiber_phi}
    $ \{F_{J}(n\vec{\lambda})\} = \{\mathrm{Poly}_1^{(n \lambda_0, \cdots ,n \lambda_N)}\} =  \left(\Lb^{|\vec{\lambda}| \cdot n}-\Lb^{|\vec{\lambda}| \cdot n - N}\right)$, where $|\vec{\lambda}|:=\sum_i \lambda_i$ . In other words, $ \{F_{J}(n\vec{\lambda})\}$ only depends on $n\vec{\lambda}$ .
    \end{prop}
    \begin{proof}
        Set $d_i:=n \lambda_i > 0$ for the notational convention. Upto $S_{N+1}$-action on $\{0,\dotsc,N\}$ (forgetting that $\lambda_0 \le \cdots \le \lambda_N)$, consider instead $F_{\left<m\right>}(\vec{d})$ with $\left<m\right> = \{ 0, \cdots, m-1 \}$ and $\vec{d} = (d_0,\cdots, d_N)$ with $|\vec{d}|:=\sum\limits_{i=0}^{N} d_i$. We now want to show that
        \begin{equation}\label{eq:Poly_modif}
        \{F_{\left<m\right>}(\vec{d})\} = \{\mathrm{Poly}_1^{(d_0, \cdots ,d_N)}\} =  \left(\Lb^{|\vec{d}|}-\Lb^{|\vec{d}|- N}\right).
        \end{equation}
        
        \medskip
        
        \par In order to prove this, we first check that if we set $d_i = 0$ for some $i \ge m$, then 
        \[
            \{F_{\left<m\right>}(\vec{d})\} = \{\mathrm{Poly}_1^{(d_0, \cdots ,d_N)}\} =  \Lb^{|\vec{d}|}.
        \]
        To see this, note that $i \not \in \left<m\right>$, so that $f_i$ is monic of degree $d_i=0$ for any $(f_0,\dotsc,f_N) \in F_{\left<m\right>}(\vec{d})$; so $f_i=1$. Therefore, the common root condition from Definition~\ref{def:Poly_modif} is vacuous, so that $\{F_{\left<m\right>}(\vec{d})\}=\Lb^{|\vec{d}|}$ (as the space of monic polynomials of degree $d$ is isomorphic to $\Ab^d$ and so is the space of polynomials of degree $<d$).
        
        \medskip

        \par We prove equation~\eqref{eq:Poly_modif} by lexicographical induction on the ordered pairs $(N,m)$ such that $N>0$ and $0 \le m<N+1$. There are two base cases to consider:
        \begin{enumerate}
            \item If $m=0$, then $\left<0\right>=\emptyset$, so that $F_{\emptyset}(\vec{d}) \cong \mathrm{Poly}_1^{(d_0,\dotsc,d_N)}=:\mathrm{Poly}_1^{\vec{d}}$ by Definition~\ref{def:poly}.
            \item If $N=1$, then $m$ is $0$ or $1$. $m=0$ follows from above. Now assume $m=1$. Then $(f_0,f_1) \in F_{\left<1\right>}(\vec{d})$ if and only if $\deg f_0 <d_0$ and $\deg f_1=d_1>0$ with $f_1$ monic. Observe that $f_0$ cannot be 0, otherwise $f_1$ has no roots while having positive degree, which is a contradiction. Since $f_0$ can be written as $a_0g_0$ for $g_0$ monic  of degree $\deg f_0$ and $a_0 \in \Gb_m$, $F_{\left<1\right>}(\vec{d})$ decomposes into the following locally closed subsets:
            \[
            F_{\left<1\right>}(\vec{d})=\bigsqcup_{l=0}^{d_0-1}\Gb_m \times F_{\emptyset}(l,d_1)=\Gb_m \times \bigsqcup_{l=0}^{d_0-1} \mathrm{Poly}_1^{(l,d_1)}.
            \]
            Therefore,
            \begingroup
            \allowdisplaybreaks
            \begin{align*}
            \{F_{\left<1\right>}(\vec{d})\}&=\{\Gb_m\}\sum_{l=0}^{d_0-1}\left\{\mathrm{Poly}_1^{(l,d_1)}\right\}=(\Lb-1)\left(\Lb^{d_1}+\sum_{l=1}^{d_0-1}(\Lb^{l+d_1}-\Lb^{l+d_1-1})\right)\\
            &=(\Lb-1)(\Lb^{d_1}+\Lb^{d_0+d_1-1}-\Lb^{d_1})=(\Lb-1)\Lb^{d_0+d_1-1}\\
            &=\Lb^{d_0+d_1}-\Lb^{d_0+d_1-1}
            \end{align*}
            \endgroup
        \end{enumerate}
        
        \medskip
        
        In general, assume that the statement is true for any $(N',m')$ whenever $N'<N$ or $N'=N$ and $m' \le m$. If $m+1<N+1$, then we want to prove the assertion for $(N,m+1)$. We can take the similar decomposition as the base case $(1,1)$, except that we vary the degree of $f_m$, which is the ($m+1$)-st term of $(f_0,\dotsc,f_N) \in F_{\left<m+1\right>}(\vec{d})$, and $f_m$ can be $0$. If $f_m=0$, then $(f_0,\dotsc,\widehat{f_m},\dotsc,f_N)$ have no common roots, so that $(f_0,\dotsc,\widehat{f_m},\dotsc,f_N) \in F_{\left<m\right>}(d_0,\dotsc,\widehat{d_m},\dotsc,d_N)$ (and vice versa). Henceforth, as a set,
        \begingroup
        \allowdisplaybreaks
        \begin{align*}
        F_{\left<m+1\right>}(\vec{d}) &= F_{\left<m\right>}(d_0,\dotsc, \widehat{d_m} ,\dotsc, d_N)\: \bigsqcup \:(\Gb_m \times F_{\left<m\right>}(d_0,\dotsc, 0 ,\dotsc, d_N))\\
        &\phantom{=}\:\:\bigsqcup \left( \Gb_m \times  \bigsqcup_{\ell=1}^{d_m-1}F_{\left<m\right>}(d_0,\dotsc, \ell ,\dotsc, d_N) \right).
        \end{align*}
        \endgroup

        \par By induction,
        
        \begingroup
        \allowdisplaybreaks
        \begin{align*}
        \left\{F_{\left<m+1\right>}(\vec{d})\right\}&=\left\{F_{\left<m\right>}(d_0,\cdots, \widehat{d_m} ,\cdots, d_N)\right\}+(\Lb-1)\left\{F_{\left<m\right>}(d_0,\cdots, 0 ,\cdots, d_N)\right\}\\
        &\phantom{=}\text{ }+(\Lb-1)\sum_{\ell=0}^{d_m-1}\left\{F_{\left<m\right>}(d_0,\cdots, \ell ,\cdots, d_N)\right\}\\
        &= \Lb^{|\vec{d}|-d_m} - \Lb^{\vec{d}-d_m-N+1}+(\Lb-1) \cdot \Lb^{|\vec{d}|-d_m}\\
        &\phantom{=}\text{ }+(\Lb-1)\sum_{\ell=1}^{d_m-1}\left( \Lb^{|\vec{d}|-d_m+\ell} - \Lb^{|\vec{d}|-d_m+\ell-N} \right)\\
        &= \Lb^{|\vec{d}|-d_m} - \Lb^{|\vec{d}|-d_m-N+1}+\Lb^{|\vec{d}|-d_m+1} - \Lb^{|\vec{d}|-d_m} \\
        &\phantom{=}\text{ }+(\Lb-1)\Lb(\Lb^{|\vec{d}|-d_m}-\Lb^{|\vec{d}|-d_m-N})(1+\Lb+\dotsb+\Lb^{d_m-2})\\
        &=\Lb^{|\vec{d}|-d_m+1}-\Lb^{|\vec{d}|-d_m-N+1}+\Lb(\Lb^{|\vec{d}|-d_m}-\Lb^{|\vec{d}|-d_m-N})(\Lb^{d_m-1}-1) \\
        &=\Lb^{|\vec{d}|-d_m+1}-\Lb^{|\vec{d}|-d_m-N+1}+\Lb^{|\vec{d}|}-\Lb^{|\vec{d}|-d_m+1}-\Lb^{|\vec{d}|-N}+\Lb^{|\vec{d}|-d_m-N+1}\\
        &=\Lb^{|\vec{d}|}-\Lb^{|\vec{d}|-N}
        \end{align*}
        \endgroup
        
    \end{proof}
    
    \par Combining \eqref{eq:sum_stratif} and Proposition~\ref{prop:fiber_phi} with $\sum\limits_{J \subsetneq \{0,\dotsc,N\}}E_J=(\Ab^{N+1}\setminus 0)$, we finally acquire
    \begingroup
    \allowdisplaybreaks
    \begin{align*}
    \{\Hom_{n}(\Pb^{1}, \Pcv)\}&=\{\Gb_m\}^{-1}\{T\}=(\Lb-1)^{-1}\sum_{J \subsetneq \{0,\dotsc,N\}}\{E_J\}\{\mathrm{Poly}_1^{(n\vec{\lambda})}\}\\
    &=(\Lb-1)^{-1}(\Lb^{N+1}-1)\{\mathrm{Poly}_1^{(n\vec{\lambda})}\}=\left(\sum\limits_{i=0}^{N} \Lb^{i}\right) \cdot  \left(\Lb^{|\vec{\lambda}| \cdot n}-\Lb^{|\vec{\lambda}| \cdot n - N}\right)
    \end{align*}
    \endgroup
    
    \par This finishes the proof of Proposition~\ref{prop:motivecount_Hom}.

    \end{proof}

    \subsection{Point count of Hom stack}\label{subsec:ptcount_Hom}

    Using Proposition~\ref{prop:motivecount_Hom}, we immediately obtain the weighted point count $ \#_q\left(\Hom_n(\Pb^1,\Pcv)\right)$:

    \begin{cor}\label{cor:maincount}
        Fix the weight $\vec{\lambda} = (\lambda_0, \dotsc, \lambda_N)$ with $|\vec{\lambda}|:=\sum\limits_{i=0}^{N} \lambda_i$. Then the weighted point count of the Hom stack $\Hom_n(\Pb^1,\Pcv)$ over $\Fb_q$ is
        \[
        \begin{array}{ll}
        \#_q\left(\Hom_n(\Pb^1,\Pcv)\right) &= \left(\sum\limits_{i=0}^{N} q^{i}\right) \cdot  \left(q^{|\vec{\lambda}|n}-q^{|\vec{\lambda}|n - N}\right) \\
        &\\
        &=q^{|\vec{\lambda}|n - N} \cdot \left(q^{2N} + \dotsb + q^{N+1} - q^{N-1} - \dotsb - 1\right)
        \end{array}
        \]
        Denote $\delta:=\gcd (\lambda_0,\dotsc,\lambda_N)$ and $\omega:= \max \gcd(\lambda_i,\lambda_j)$ for $ 0 \le i,j \le N$. Then the number $|\Hom_n(\Pb^1,\Pcv)(\Fb_q)/\sim|$ of $\Fb_q$--isomorphism classes of $\Fb_q$--points (i.e., the non--weighted point count over $\Fb_q$) of $\Hom_n(\Pb^1,\Pcv)$ satisfies \[\delta \cdot \#_q\left(\Hom_n(\Pb^1,\Pcv)\right) \le \left|\Hom_n(\Pb^1,\Pcv)(\Fb_q)/\sim\right| \le \omega \cdot \#_q\left(\Hom_n(\Pb^1,\Pcv)\right)\;.\]
    \end{cor}
    
    \begin{proof}
        The first part of the Corollary follows as $\#_q:K_0(\mathrm{Stck}_{\Fb_q}) \rightarrow \Q$ is a ring homomorphism with $\#_q(\Lb)=q$ as $\Lb=\{\Ab_{\Fb_q}^1\}$ . For the second part, notice that for each $\varphi \in \Hom_n(\Pb^1,\Pcv)(\Fb_q)/\sim$ , it contributes 1 towards $|\Hom_n(\Pb^1,\Pcv)(\Fb_q)/\sim|$ instead of $\frac{1}{|\mathrm{Aut}(\varphi)|}$ for $\#_q(\Hom_n(\Pb^1,\Pcv))$. 
        Thus, we need to check that for any $\varphi \in \Hom_n(\Pb^1,\Pcv)(\Fb_q)$ with $\delta:=\gcd (\lambda_0,\dotsc,\lambda_N)$ and $\omega:= \max \gcd(\lambda_i,\lambda_j)$ for $ 0 \le i,j \le N$, the automorphism group satisfies the following :
        \[
        \delta \le |\mathrm{Aut}(\varphi)| \le \omega \; .
        \]
        By Proposition~\ref{prop:DMstack}, we can represent $\varphi$ as a tuple $(f_0,\dotsc,f_N)$ of sections $f_i \in H^0(\Oc_{\Pb^1_{\Fb_q}}(n\lambda_i))$, with equivalence relation given by a $\Gb_m$--action. Since the automorphism group of $\varphi$ is identified with the subgroup of $\Gb_m$ fixing $(f_0,\dotsc,f_N)$, $\mathrm{Aut}(\varphi)$ consists of $u \in \Gb_m(\Fb_q)$ such that $u^{\lambda_i}f_i=f_i$ for any $i$ . Since $f_i$'s have no common root and the degree of the morphism $\varphi$ is $n \in \N$, at least two of those are nonzero; call $I$ to be the set of $i$'s with $f_i \neq 0$ . Then, $u^{\lambda_i}=1$ for any $i \in I$, so that $u$ is a $\mathrm{gcd}(\lambda_i \; : \; i \in I)^{\mathrm{th}}$ root of unity. This shows that $\mathrm{Aut}(\varphi)$ is a finite cyclic group of order $\mathrm{gcd}(\lambda_i \; : \; i \in I)$, proving the second part of the Corollary.
    \end{proof}
    
    Above proof shows that computing automorphism groups of $\Fb_q$--points of $\Hom_n(\Pb^1,\Pcv)$ is the key ingredient for comparing between weighted and non--weighted point counts over $\Fb_q$. Instead, we bypass this issue by using properties of the inertia stack $\Ic\left(\Hom_n(\Pb^1,\Pcv)\right)$, particularly by using Theorem~\ref{thm:ptcounts} and Proposition~\ref{prop:Hom_inertia_decomp}.

    \medskip

    \subsection{Point count of Inertia of Hom stack}\label{subsec:motive_inertia}

    \par We compute the class $\left\{ \Ic\left(\Hom_n(\Pb_K^1,\Pc_K(\vec\lambda))\right) \right\}$, which renders the non--weighted point count of the moduli stack $\Lc_{g, |\Delta_{g}| \cdot n}$ over $\Fb_q$.

    \begin{prop}\label{prop:motive_inertia}
        Take the same notation as in Proposition~\ref{prop:wtproj_inertia_decomp}. Then,
        \begin{align*}
            \left\{ \Ic\left(\Hom_n(\Pb_K^1,\Pc_K(\vec\lambda))\right) \right\} = \sum_{g \in |(\Gb_m)_K|} \left\{ \Hom_n(\Pb_{\kappa(g)}^1,\Pc_{\kappa(g)}(\vec\lambda_{I_g})) \right\}
        \end{align*}
    \end{prop}
    \begin{proof}
        This directly follows from Proposition~\ref{prop:Hom_inertia_decomp} and the definition of the Grothendieck Ring.
    \end{proof}

    \par Above definition combined with the proof of Corollary~\ref{cor:maincount} gives an algorithm for computing $|\Hom_n(\Pb^1,\Pcv)(\Fb_q)/\sim|$:
    
    \begin{prop}\label{prop:algo_ptcount}
        Take the same notation as in Proposition~\ref{prop:Hom_inertia_decomp} for $K=\Fb_q$, and let $R$ be the set of positive integers $r$ (including $r=1$) that divides $q-1$. Then,
        \[
            |\Hom_n(\Pb^1,\Pcv)(\Fb_q)/\sim|=\sum_{r \in R} \varphi(r)\cdot \#_q(\Hom_n(\Pb^1,\Pc(\vec\lambda_r)))
        \]
        where $\varphi$ is the Euler $\varphi$-function and $\vec\lambda_r:=\vec\lambda_{I_{\zeta_r}}$ for some primitive $r^{\mathrm{th}}$ root of unity $\zeta_r$ in $\Gb_m$.
    \end{prop}
    \begin{proof}
        Recall that the multiplicative group $\Fb_q^*$ of a finite field $\Fb_q$ is a cyclic group of order $q-1$. By the primitive root condition, we see that $\zeta_r \in \Gb_m(\Fb_q)$ iff $r | (q-1)$. This implies that the substack $\Hom_n(\Pb^1_{\kappa(\zeta_r)},\Pc_{\kappa(\zeta_r)}(\vec\lambda_r))$ contributes $\Fb_q$--rational points iff $r$ divides $q-1$, hence the definition of the set $R$. As $\vec\lambda_r$ is independent of the choice of a primitive $r^{\mathrm{th}}$ root of unity and there are $\varphi(r)$ number of them, simplifying $\#_q\left\{ \Ic\left(\Hom_n(\Pb_K^1,\Pc_K(\vec\lambda))\right) \right\}$ gives the desired formula by Theorem~\ref{thm:ptcounts}.
    \end{proof}

    \begin{rmk}\label{rmk:modular_arithmetic}
    Note that writing a closed-form formula for $|\Hom_n(\Pb^1,\Pcv)(\Fb_q)/\sim|$ is difficult in general, as Euler $\varphi$-function is used, the sum is over all possible positive factors of $q-1$, and the length of $\vec\lambda_r$ can vary. Nevertheless, it is possible to obtain a closed-form formula by hand for special cases with mild assumptions on $q$ (Theorem~\ref{thm:low_genus_count} is a good example).
    \end{rmk}

    \section{Moduli stack $\Lc_{g, |\Delta_{g}| \cdot n}$ of quasi--admissible odd--degree hyperelliptic genus $g$ fibrations over $\Pb^{1}$}\label{sec:qadm_fib}
    

    \par In this section, we first define a rational fibration with a marked section, which allows us to define a hyperelliptic genus $g$ fibration with a marked Weierstrass section as a double cover fibration. Subsequently, we focus on a quasi--admissible hyperelliptic genus $g$ fibration over $\Pb^1$ with a marked Weierstrass section which extends a family of odd--degree hyperelliptic genus $g$ curves over $\Fb_q(t)$ with a marked Weierstrass point. For detailed references on hyperelliptic fibrations or fibered algebraic surfaces (over an algebraically closed field), we refer the reader to \cite{Liu, Liedtke}.

    \medskip

    \par Recall that a hyperelliptic curve $C$ is a separable morphism $\phi : C \to \Pb^1$ of degree 2. In order to extend the notion of hyperelliptic curve $C$ into family, we first generalize the notion of rational curve $\Pb^1$ into family.
    
    \begin{defn}\label{def:rational_fib}
        A \emph{rational fibration with a marked section} is given by a flat proper morphism $h:H \rightarrow \Pb^1$ of pure relative dimension 1 with a marked section $s':\Pb^1 \rightarrow H$ such that
        \begin{enumerate}
            \item any geometric fiber $h^{-1}(c)$ is a connected rational curve (so that arithmetic genus is 0),
            \item $s'(\Pb^1)$ is away from the non-reduced locus of any geometric fiber, and
            \item $s'(\Pb^1)$ is away from the singular locus of $H$.
        \end{enumerate}
        If the geometric generic fiber of $h$ is a smooth rational curve, then we call $(H,h,s')$ a \emph{$\Pb^1$-fibration}.
    \end{defn}
    
    \par We will occationally call $(H,h,s')$ a \emph{rational fibration} when there is no ambiguity on the marked section $s'$. Note that we allow a rational fibration $H$ to be reducible (when generic fiber is a nodal chain), and the total space of a $\Pb^1$-fibration can be singular. Certain double cover of the rational fibration gives us the hyperelliptic genus $g$ fibration with a marked Weierstrass section.

    \begin{defn}\label{def:hypell_fib}
        A \emph{hyperelliptic genus $g$ fibration with a marked Weierstrass section} consists of a tuple $(X,H,h,f,s,s')$ of a rational fibration $h: H \rightarrow \Pb^1$, a flat proper morphism $f : X \to H$ of degree 2 with $X$ connected and reduced, and sections $s: \Pb^1 \to X$ and $s':\Pb^1 \to H$ such that
        \begin{enumerate}
            \item Each geometric fiber $(h \circ f)^{-1}(c)$ is a connected 1--dimensional scheme of arithmetic genus $g$,
            \item $s(\Pb^1)$ is contained in the smooth locus of $h \circ f$ and is away from the non-reduced locus of any geometric fiber,
            \item $s'=f \circ s$ and $s(\Pb^1)$ is a connected component of the ramification locus of $f$ (i.e., $s'(\Pb^1)$ is a connected component of the branch locus of $f$), and
            \item if $p$ is a node of a geometric fiber $h^{-1}(c)$, then any $q \in f^{-1}(p)$ is a node of the fiber $(h\circ f)^{-1}(c)$,
            \item if the branch divisor of $f$ contains a node $e$ of a fiber $h^{-1}(t)$ with $t$ a closed geometric point of $\Pb^1$, then the branch divisor contains either an irreducible component of $h^{-1}(t)$ containing $e$ or an irreducible component of the singular locus of $H$ containing $e$.
        \end{enumerate}
        The \emph{underlying genus $g$ fibration} is a tuple $(\pi:=h \circ f, s)$ with $\pi:X \to \Pb^1$ a flat proper morphism with geometric fibers of arithmetic genus $g$ with a marked Weierstrass point given by $s$.
    \end{defn}
    
    \begin{note}\label{note:aut_fib}
        An isomorphism between hyperelliptic genus g fibrations $(X_1,H_1,h_1,f_1,s_1,s_1')$ and $(X_2,H_2,h_2,f_2,s_2,s_2')$ is given by a pair of isomorphisms $\alpha:X_1 \rightarrow X_2$ and $\beta:H_1 \rightarrow H_2$ such that
        \begin{enumerate}
            \item $h_2 \circ \beta = h_1$ and $f_2 \circ \alpha = \beta \circ f_1$ ($\Pb^1$--isomorphism criteria), and
            \item $\beta \circ s=s'$ (compatibility with sections).
        \end{enumerate}
    \end{note}
    
    \par From now on, we only consider non-isotrivial hyperelliptic fibrations, i.e., the underlying genus $g$ fibrations must be non-isotrivial. Thus, non-isotrivialness will be assumed on every statement and discussions below.
    
    \medskip

    \par Recall that a fibration with a section is said to be \textit{stable} if all of its fibers are stable pointed curves. This leads to the following definition in the hyperelliptic case:

    \begin{defn}\label{def:qstb_hypfib}
        A \emph{stable hyperelliptic genus $g$ fibration with a marked Weierstrass section} is a hyperelliptic genus $g$ fibration $(X,H,h,f,s,s')$ with $K_X+s(\Pb^1)$ is $\pi$-ample. We assume that $X$ is not isotrivial, i.e., the trivial hyperelliptic fiber bundle over $\Pb^{1}$ with no singular fibers.
        
        Moreover, if the geometric generic fiber is smooth, then $(X,H,h,f,s,s')$ is called a \emph{stable odd hyperelliptic genus $g$ model over $\Pb^1$}.
    \end{defn}
    
    \par Conditions in the above definition implies that $(X,s(\Pb^1))/\Pb^1$ is log canonical. In classical language, this means that there are no smooth rational curves of self-intersection $-1$ and $-2$ in a fiber without meeting $s(\Pb^1)$.

    \begin{exmp}\label{eg:sstab_hyp_sing_P1}
        Suppose that $(X,H,h,f,s,s')$ is a stable odd hyperelliptic genus $g$ model with a marked Weierstrass section. Then, it is possible that $f:X \rightarrow H$ in a \'etale local neighborhood of $p \in H$ is the map $\Ab^2_{x,y} \rightarrow \Ab^2_{x,y}/\mu_2$, where $\mu_2$ acts on $\Ab^2_{x,y}$ by $(x,y) \mapsto (-x,-y)$ . In this case, $\pi$ can be given by $\Ab^2_{x,y} \rightarrow \Ab^1_{z}$ by $z=xy$ . Note that $H$ admits an $A_1$--singularity at $p$, $f^{-1}(p)$ is a node of a fiber of $\pi$, but $X$ is nonsingular. In general, $X$ and $H$ admit at worst $A_l$--singularities for some $l$ (because geometric fibers of $X$ are nodal curves), where $A_u$--singularities of surfaces are \'etale locally given by $w^2+x^2+y^{u+1}=0$. This follows from the fact that 1--parameter deformation of nodes create such singularities. Note that on the neighborhood of such an isolated singular point of $H$, the branch locus of $f$ is concentrated at the point if it contains the point, which only appears possibly at singular points of the fibers of $h:H \rightarrow \Pb^1$.
    \end{exmp}
    
    \begin{exmp}\label{eg:branch_div}
        Suppose that $(X,H,h,f,s,s')$ is a stable odd hyperelliptic genus $g$ model with a marked Weierstrass section over a field $K$. The goal is to classify singularities of the branch divisor of $f$. By the definition, the branch divisor decomposes into $B \sqcup s'(\Pb^1_K)$, which is contained in the smooth locus of $H$ by the definition. First, consider a geometric point $c$ in the intersection $B \cap H_t$, where $t$ is a geometric point of $\Pb^1_K$ and $H_t$ is the fiber $h^{-1}(t)$. Since the corresponding double cover $X_t$ (which is a fiber over $t$ of $h \circ f$) only admits nodes as singularities, the multiplicity $m$ of $B \cap H_t$ at $c$ is at most 2, as $f_t : X_t \rightarrow Y_t$ \'etale locally near $c$ is given by the equation
        \[
            \Spec(\overline{K}[y,z]/(z^2-y^m)) \rightarrow \Spec(\overline{K}[y]) \text{, where $y$ is the uniformizer of $c \in H_t$.}
        \]
        Since $B$ does not contain any irreducible component of geometric fibers of $h$ (as any geometric fiber of $h \circ f$ is reduced), above implies that the multiplicity of $B$ at any geometric point is at most 2. Thus, the support of $B$ possibly admits plane double point curve singularities, \'etale locally of the form $y^2-x^m=0$ with $m \in \N_{\ge 2}$, on the geometrically reduced locus of $B$, and is smooth elsewhere. Those singularities are in fact $A_{m-1}$ (curve) singularities.
    \end{exmp}
    
    \par Example~\ref{eg:sstab_hyp_sing_P1} and \ref{eg:branch_div} illustrate that a general stable odd hyperelliptic genus $g$ model often gives a mildly singular $\Pb^1$-fibration and mildly singular branch divisor on it. On the other hand, we could instead consider the hyperelliptic fibrations with smooth $\Pb^1$-bundle $H$, but with $X$ and the branch divisor having worse singularities. Then, each fiber of $X$ is irreducible and is a double cover of $\Pb^1$ branched over $2g+2$ number of points, where many of these points could collide. For instance, if $l$ branch points collide, then the preimage has $A_{l-1}$-singularity on the fiber, given \'etale locally by an equation $y^2-x^l=0$ . Such a curve is called the quasi--admissible hyperelliptic curve, defined in Definition~\ref{def:qadm_hyp}. Quasi--admissible hyperelliptic curves over $\Pb^1_K$ (which are non-isotrivial) are equivalent to the following:
    
    \begin{defn}\label{def:qadm_fib}
        A hyperelliptic fibration $(X,H,h,f,s,s')$ is \emph{quasi--admissible} if for every geometric point $c \in C$, $f$ restricted to the fibers of $X$ and $H$ is quasi--admissible.
        We assume that $X$ is not isotrivial over $\Pb^1$, i.e., all geometric fibers are isomorphic.
    \end{defn}
    
    \begin{rmk}\label{rmk:extension_scheme}
        Observe that the Definitions~\ref{def:rational_fib}, \ref{def:hypell_fib}, \ref{def:qstb_hypfib}, and \ref{def:qadm_fib} should be interpreted as rational / hyperelliptic / stable  / quasi--admissible curves over $\Pb^1_K$, instead of a point $\Spec~K$ (just as in Definition~\ref{def:qadm_hyp}). Thus, these definitions can be extended to corresponding curves over a general scheme $T$, assuming that any geometric point $t$ of $T$ has the property that the characteristic of the residue field is 0 or larger than $2g+1$ (when instead $g=1$, the standard definition of semistable over $T$ is more delicate whenever the characteristic of geometric point is 2 or 3, and is not analogous to the definitions proposed in this paper).
    \end{rmk}
    
    \par In particular, a quasi--admissible hyperelliptic fibration $(X,H,h,f,s,s')$ has the property that $H$ is a $\Pb^1$-bundle, and on each geometric fiber of $H$, each point of the branch divisor away from $s'$ has the multiplicity at most $2g$ . Moreover, $X$ is the double cover of $H$ branched along the branch divisor (which coincides with the branch locus).

   \medskip
    
    \par To parameterize such fibrations, we first consider the moduli stack $\Hc_{2g}[2g-1]$ of quasi--admissible hyperelliptic genus $g$ curves characterized by \cite[Proposition 4.2(1)]{Fedorchuk} :
    
    \begin{prop}\label{prop:moduli_qadm_curve}
        If $p:=\mathrm{char}(K)$ is $0$ or $> 2g+1$, then the moduli stack $\Hc_{2g}[2g-1]$ of quasi--admissible hyperelliptic genus $g$ curves is a tame Deligne--Mumford stack isomorphic to $\Pc(4,6,8,\dotsc,4g+2)$, where a point $(a_4,a_6,a_8,\dotsc,a_{4g+2})$ of $\Pc(4,6,8,\dotsc,4g+2)$ corresponds to the quasi--admissible hyperelliptic genus $g$ curve with the Weierstrass equation
        \begin{equation}\label{eqn:qadm}
        y^2=x^{2g+1}+a_4x^{2g-1}+a_6x^{2g-2}+a_8x^{2g-3}+\dotsb+a_{4g+2}
        \end{equation}
    \end{prop}
    
    \begin{proof}
        Proof of \cite[Proposition 4.2(1)]{Fedorchuk} is originally done when $p=0$, so it suffices to show that the proof in loc.cit. extends to the case when $p>2g+1$ .
        
        When $p=0$, the proof of loc.cit. shows that the quasi--admissible hyperelliptic curves are characterized by the base $\Pb^1$ with the branch locus of degree $2g+1$ on $\Ab^1=\Pb^1\setminus \infty$, of the form
        \[
            x^{2g+1}+a_2x^{2g}+a_4x^{2g-1}+a_6x^{2g-2}+a_8x^{2g-3}+\dotsb+a_{4g+2}=0
        \]
        where $a_2=0$ and not all of the rest of $a_i$'s vanish. When $p>2g+1$, any monic polynomial of degree $2g+1$ with not all roots being identical can be written in the same way (via same method) by replacing $x$ by $x-\frac{a_2}{(2g+1)}$ (this is allowed as $2g+1<p$ is invertible). Hence, the moduli stack is indeed isomorphic to $\Pc(4,6,8,\dotsc,4g+2)$, with $a_{2i}$'s referring to the standard coordinates of $\Pc(4,6,8,\dotsc,4g+2)$ of degree $2i$ .
        
        Since $p>2g+1$. $\Pc(4,6,8,\dotsc,4g+2)$ is tame Deligne--Mumford by Proposition~\ref{prop:wtprojtame} .
    \end{proof}
    
    We are now ready to prove Proposition~\ref{prop:moduli_qadm_fib}.

    \begin{proof}[Proof of Proposition~\ref{prop:moduli_qadm_fib}]

        By the definition of the universal family $p$, any quasi--admissible hyperelliptic genus $g$ fibration $f: Y \rightarrow \Pb^1$ comes from a morphism $\varphi_f:\Pb^1 \rightarrow \Hc_{2g}[2g-1]$ and vice versa. As this correspondence also works in families, the moduli stack $\Lc_g$ is a substack of $\Hom(\Pb^1,\Hc_{2g}[2g-1])$. As $\Hc_{2g}[2g-1]$ is tame Deligne--Mumford by Proposition~\ref{prop:moduli_qadm_curve}, the Hom stack $\Hom(\Pb^{1},\Hc_{2g}[2g-1])$ is Deligne--Mumford by \cite{Olsson}. Tameness follows from \cite{AOV}, as $\Hc_{2g}[2g-1]$ itself is tame. Thus, $\Lc_g$ is tame Deligne--Mumford as well.
        
        Since any quasi--admissible hyperelliptic genus $g$ fibration $f$ is not isotrivial, $\varphi_f$ must be a non-trivial morphism, i.e., the image of $f$ in $\Hc_{2g}[2g-1])$ is 1-dimensional. Since non-trivialness of a morphism is an clopen condition, the corresponding clopen locus (consisting of the union of connected components) $\Hom_{>0}(\Pb^1,\Hc_{2g}[2g-1])$ is indeed isomorphic to $\Lc_g$.
    \end{proof}

    \par We now have the following arithmetic invariant of the moduli stack $\Lc_{g, |\Delta_{g}| \cdot n}$ over $\Fb_q$.

    \begin{cor} [Motive and weighted point count of $\Lc_{g, |\Delta_{g}| \cdot n}$ over $\Fb_q$]
        \label{cor:pointcount}
        If $\mathrm{char}(K) = 0$ or $> 2g+1$, then the motive of $\Lc_{g, |\Delta_{g}| \cdot n}$ in the Grothendieck ring of $K$--stacks $K_0(\mathrm{Stck}_{K})$ is equivalent to
        \begingroup
        \allowdisplaybreaks
        \begin{align*}
        \left[\Lc_{g, |\Delta_{g}| \cdot n}\right] &= \left( \sum\limits_{i=0}^{2g-1} \Lb^{i} \right) \cdot \left(\Lb^{|\vec{\lambda_g}|n}-\Lb^{|\vec{\lambda_g}|n - 2g+1} \right)\\
        &= \Lb^{2g(2g+3)n} \cdot (\Lb^{2g-1} + \Lb^{2g-2} + \dotsb + \Lb^2 + \Lb^1 -\Lb^{-1}-\Lb^{-2}- \dotsb - \Lb^{-2g+2} - \Lb^{-2g+1}).
        \end{align*}
        \endgroup
        If $K=\Fb_q$ with $\mathrm{char}(\Fb_q)> 2g+1$, then \[\#_q\left(\Lc_{g, |\Delta_{g}| \cdot n}\right) = q^{2g(2g+3)n} \cdot (q^{2g-1} + q^{2g-2} + \dotsb + q^2 + q^1 -q^{-1} -q^{-2} - \dotsb - q^{-2g+2} - q^{-2g+1}) \;.\]
    \end{cor}
    
    \begin{proof}
    This follows from combining Proposition~\ref{prop:moduli_qadm_curve} and Proposition~\ref{prop:moduli_qadm_fib} with Corollary~\ref{cor:maincount}.        
    \end{proof}

    Explicitly via the birational transformation from one family of curves to another, we construct a geometric transformation from $\mathcal{S}_g(K)$ the $K$--points of the moduli functor $\mathcal{S}_g$ of the stable odd hyperelliptic genus $g \ge 2$ models (see Definition~\ref{def:qstb_hypfib}) over $\Pb^1$ with a marked Weierstrass point to $\Lc_g(K)$ the $K$--points of the moduli functor $\Lc_g \coloneqq \mathrm{Hom}(\Pb^1,\Hc_{2g}[2g-1] \cong \Pc(4,6,8,\dotsc,4g+2))$. In fact, this transformation is injective as Theorem~\ref{thm:hypell_surf_and_qadm_fib} shows.

    \begin{proof}[Proof of Theorem~\ref{thm:hypell_surf_and_qadm_fib}]

    The key idea of proof is to construct $\Fc$ by using relative canonical model, a particular birational transformation from the subject of relative minimal model program. We prove this in a few steps, beginning with a preliminary step. We construct and verify properties of $\Fc$ in the other steps:
    
    \medskip
    
    \noindent \textbf{Step 1. Log canonical singularities and log canonical models}. The main reference here is \cite{Fujino} when $\mathrm{char}(K)=0$, and \cite[\S 5--6]{Tanaka} when $\mathrm{char}(K) \neq 0$, noting that both references deal with algebraically closed fields instead.
    
    \par First, we need the following definition for types of singularities of a pair $(S,D)$ of a normal $\overline{K}$--surface $S$ and an effective $\Rb$--divisor $D$ on $S$:
    
    \begin{defn}(\cite[\S 2.4]{Fujino}, \cite[Definition 5.1]{Tanaka}) \label{def:lc_sing}
        A pair $(S, D)$ is log canonical if
        \begin{enumerate}
            \item the log canonical divisor $K_S + D$ is $\Rb$--Cartier,
            \item for any proper birational morphism $\pi : W \rightarrow S$ and the divisor $D_W$ defined by
            \[
                K_W+D_W=\pi^*(K_S+D),
            \]
            then $D_W \le 1$, i.e., when writing $D_W=\sum_i a_iE_i$ as a sum of distinct irreducible divisors $E_i$, $a_i \le 1$ for every $i$.
        \end{enumerate}
        Moreover, if a pair $(S, D)$ is defined over a non-algebraically closed field $K$, then it is called log canonical if its base-change to $\overline{K}$ is.
    \end{defn}
    
    For instance, if $S$ is smooth and $D$ is a reduced simple normal crossing divisor, then $(S, D)$ is log canonical. Similarly, if $w \in \Rb \cap [0,1]$, then $(S, wD)$ is log canonical under the same assumptions. Note that we cannot consider $w>1$ under the same assumptions, as the weight on each irreducible component of $D$ must be at most 1.
    
    For example, consider a stable odd hyperelliptic genus $g$ model $(X,H,h,f,s,s')$ over $K$, consider the pair $(H_{\overline{K}}, wB_{\overline{K}} + (s'(\Pb^1_K))_{\overline{K}})$ defined over $\overline{K}$ where the branch divisor of $h$ decomposes as $B \sqcup s'(\Pb^1_K)$ and $w \in \Rb \cap (0,1/2]$ is a weight (since $B$ can have components of multiplicity 2 by Example~\ref{eg:branch_div}, we consider weights at most $1/2$). To claim that this pair is log canonical under additional condition on $w$, it suffices to consider neighborhoods of singular points of $H_{\overline{K}}$ and support of $B_{\overline{K}}$ by the above observation.
    
    First, recall that the isolated singularities of $H_{\overline{K}}$ away from the support of $wB_{\overline{K}} + (s'(\Pb^1_K))_{\overline{K}}$ is of type $A_{l'}$ for some $l'$ by Example~\ref{eg:sstab_hyp_sing_P1}. Hence, the pair is log canonical at neighborhoods of such points (in fact, those points are called canonical singular points of $H_{\overline{K}}$). Also, at a singular point $c$ of the support of $B_{\overline{K}}$, $H_{\overline{K}}$ is smooth and $B_{\overline{K}}$ is reduced at $c$ but $B_{\overline{K}}$ admits $A_l$--singularities by Example~\ref{eg:branch_div}. Therefore, the pair has log canonical singularities whenever $w \le \frac12 + \frac{1}{l+1}$ by \cite{Jarvilehto} (summariezd in \cite[Introduction]{GHM}, where the log canonical threshold is the supremum of values $w$ that makes the pair log canonical).
    
    \medskip
    
    To construct a log canonical model, consider a pair $(S, D)$ as the beginning of this step with projective $\overline{K}$--morphism $f : S \rightarrow C$ into a $\overline{K}$--variety $C$, and assume that $D$ is $\Qb$--divisor and $S$ is $\Qb$--factorial. If $(S,D)$ is log canonical with $K_S+D$ not $f$--antinef, then \cite[Pages 1750--1751]{HP} uses key birational geometry results from \cite{Fujino, Tanaka} to construct the $f$--log canonical model, defined below. In fact, analogous arguments from \cite[Proof of Proposition 11]{HP} implies that the same procedure can be applied to $f:(S, D) \rightarrow C$ over a field $K$, leading to the following definition:
    
    \begin{defn}\label{def:rel_lc_program}
        Suppose that $(S, D)$ is a log canonical pair over a field $K$ where $S$ is a normal projective $\Qb$--factorial surface and $D$ is a $\Qb$--divisor. Assume that $f: S \rightarrow C$ is a projective morphism into a $K$--variety $C$ with $K_S+D$ not $f$--antinef. If $K$ is algebraically closed, then the $f$--log canonical model is a pair $(S',D')$ with a projective morphism $f':S' \rightarrow C$, where
        \[
            S':=\underline{\Proj}\bigoplus_{m \in \N}f_*\Oc_{S}(m(K_S+D))
        \]
        and $D':=\phi_*D$ where $\phi:S \rightarrow S'$ is the induced birational morphism.
        
        If $K$ is not algebraically closed, then the $f$--log canonical model is the $\Gal(\overline{K}/K)$--fixed locus of the $f_{\overline{K}}$--log canonical model of $(S_{\overline{K}},D_{\overline{K}})$.
    \end{defn}

    \medskip
    
    \noindent \textbf{Step 2. Construction of faithful $\mathcal F:\mathcal{S}_g(K) \rightarrow \Lc_g(K)$}. Fix any member of $\mathcal{S}_g(K)$, i.e., a stable odd hyperelliptic genus $g$ model $(X,H,h,f,s,s')$ . Denote $B\sqcup s'(\Pb^1_K)$ to be the divisorial part of the branch locus of $f:X \rightarrow H$ ($B$ is also called branch divisor in literature). Notice that $h$ restricted to $B$ has degree $2g+1$. By Step 1 above, $(H,\frac{1}{2g}B+s'(C))$ is log canonical. Take $h$--log canonical model of $(H,\frac{1}{2g}B+s'(C))/\Pb^1_K$ to obtain a birational $\Pb_K^1$--morphism $\varphi: (H,\frac{1}{2g}B+s'(\Pb^1_K)) \rightarrow (H',D')$ where $H'$ is a rational fibration over $K$ and $D'$ is a $\Rb$--divisor of $H'$ defined over $K$ (c.f. Definition~\ref{def:rel_lc_program}). Since the only canonical rational curve, defined over an algebraically closed field with $\frac{1}{2g}$ weights on $(2g+1)$ points and weight 1 on another point, is a smooth rational curve where the point of weight 1 is distinct from the other points (of weight $\frac{1}{2g}$), $H'$ is a $\Pb^1$--bundle (given by $h':H' \rightarrow \Pb^1_K$). This description shows that $D'$ decomposes into $\frac{1}{2g}A'+T'$ where $A'$ is a divisor of $H'$ and $T'$ consists of weight 1 points on each geometric fiber of $H'/\Pb_K^1$ . Thus, $T'$ comes from a section $t'$ of $h'$ . We will show that $H'$ is the $\Pb^1$--fibration associated to the desired quasi--admissible hyperelliptic genus $g$ fibration.
        
    To finish the construction of the quasi--admissible fibration, take Stein factorization on $\varphi \circ f$. This gives a finite morphism $f':X' \rightarrow H'$ and a morphism $\psi:X \rightarrow X'$ with geometrically connected fibers such that $\varphi \circ f=f' \circ \psi$ . Since $f$ is finite of degree 2 and $\varphi$ is birational, $f'$ is finite of degree 2 and $\psi$ is birational. Moreover, $B':=A'+T'$ is the branch locus of $f'$ . By calling $t$ to be the unique lift of $t'$ on $h' \circ f'$, $(X',H',h',f',t,t')$ is the desired quasi--admissible hyperelliptic fibration. Define $\mathcal F(X,H,h,f,s,s'):=(X',H',h',f',t,t')$ .
        
    \medskip
        
    To see that $\mathcal F$ is faithful, suppose that there are two isomorphisms 
        \[
        (\alpha_i,\beta_i):(X_1,H_1,h_1,f_1,s_1,s_1') \rightarrow (X_2,H_2,h_2,f_2,s_2,s_2')
        \]
    between stable odd hyperelliptic genus $g$ models that induce the same isomorphism under $\mathcal F$ :
        \[
        (\alpha',\beta'):\mathcal F(X_1,H_1,h_1,f_1,s_1,s_1') \rightarrow \mathcal F(X_2,H_2,h_2,f_2,s_2,s_2')
        \]
    Denote $(X_j',H_j',h_j',f_j',t_j,t_j')=\mathcal F(X_j,H_j,h_j,f_j,s_j,s'_j)$ for $j=1,2$ . From the construction of $\mathcal F$ shown above, induced morphisms $X_j \rightarrow X_j'$ and $H_j \rightarrow H_j'$ are birational for each $j$. Since they are separated varieties over $K$, $(\alpha_1,\beta_1)$ must be equal to $(\alpha_2,\beta_2)$ , hence $\mathcal F$ is faithful.
        
    \medskip

    \noindent \textbf{Step 3. Fullness of $\mathcal F$}. Given any isomorphism $\psi$ between $( X_i' , H_i' , h_i' , f_i' , t_i , t_i')$'s in $\Lc_g(K)$ as images of $( X_i , H_i , h_i , f_i , s_i , s'_i ) \in \mathcal{S}_g(K)$ under $\Fc$, notice that $h_i'$'s and $h_i' \circ f_i'$'s have smooth geometric generic fibers for $i=1,2$ and $\psi$ comes in pairs of isomorphisms $\psi_1:X_1' \rightarrow X_2'$ and $\psi_2:H_1' \rightarrow H_2'$ (so denote $\psi=(\psi_1,\psi_2)$). Then, $\psi$ lifts to a pair of birational maps $\overline{\psi}=(\overline{\psi}_1,\overline{\psi}_2)$ between $X_i$'s and $H_i$'s which induce isomorphisms on geometric generic fibers and irreducible components of any geometric fiber meeting the sections $s_i$'s or $s_i'$'s. To claim that those extend to isomorphisms that induce $\psi_i$'s, it suffices to understand geometric properties of related moduli stacks, as we claim that $\psi_i$'s can be interpreted as an element of $Isom$ spaces of such stacks.
    
    \par Observe first that for each $i=1,2$, $H_i$ is a $\Zb/2\Zb$--quotient of $X_i$, and $K_{X_i}+s_i(\Pb^1_K)$ is ample over $\Pb^1_K$ by the defintion. Since the branch divisor of $f_i$ is $B_i \sqcup s_i'(\Pb^1_K)$, the log canonical divisor $K_{H_i}+\frac12B_i+ s'_i(\Pb^1_K)$ is also ample over $\Pb^1_K$ as $f_i^*(K_{H_i}+\frac12B_i+ s'_i(\Pb^1_K))=K_{X_i}+s_i(\Pb^1_K)$. Since $X_i$ admits nodes as the only singularities of geometric fibers, $B_i$ on each fiber has multiplicity at most $2$ at any $\overline{K}$--points in the support. Therefore, fibers of the pair $(H_i,\frac{1}{2}B_i+s'_i(\Pb^1_K))$ are $((\frac12,2g+1),(1,1))$--stable curves of genus $0$ in the sense of \cite[\S 2.1.3]{Hassett}, meaning that $H_i$ for each $i$ is a family of such curves over $\Pb^1_K$. Note that the moduli stack $\Mg_{0,((\frac12,2g+1),(1,1))}$ of $((\frac12,2g+1),(1,1))$--stable curves of genus $0$ is a proper (so separated) Deligne--Mumford stack (it easily follows from loc.cit. and \cite[Theorem 2.1]{Hassett}), and $H_i$ is realized as $\alpha_i:\Pb^1_K \rightarrow \Mg_{0,((\frac12,2g+1),(1,1))}$. Since there is a nonempty open subset $U \subset \Pb^1_K$ such that $\psi_2$ induces an isomorphism between $h_i^{-1}(U)$'s, $\overline{\psi}_2$ is an element of $Isom_{\Mg_{0,((\frac12,2g+1),(1,1))}}(\alpha_1,\alpha_2)(U)$. Then, separatedness of $\Mg_{0,((\frac12,2g+1),(1,1))}$ implies that $\overline{\psi_2}$ extends to an isomorphism between $H_i$'s.
    
    \par Similar argument shows that $\psi_1$ also extends to an isomorphism between $X_i$'s (as $\overline{H}_{g,\underline{1}} \subset \Mg_{g,\underline{1}}$ is a separated Deligne--Mumford stack by \cite{Knudsen}), so it suffices to show that $\overline{\psi}_i$'s commute with $f_i$'s and induce $\psi$. The commutativity of $\overline{\psi}_i$'s follows because $H_i$'s are $\Zb/2\Zb$--quotients of $X_i$'s and any isomorphism between families of stable hyperelliptic curves with marked Weierstrass points commute with $\Zb/2\Zb$--actions. Because the birational morphisms $X_i \rightarrow X_i'$ and $H_i \rightarrow H_i'$ for any $i$ contract all but irreducible components of fibers over $\Pb^1_K$ meeting marked sections, $\overline{\psi}:=(\overline{\psi}_1,\overline{\psi}_2)$ induce $\psi$ as well. Henceforth, $\overline{\psi}$ maps to $\psi$ under $\Fc$, proving that $\Fc$ is full. 
    \end{proof}

    \begin{rmk}\label{rmk:nonextension_monomorphism}
        Due to log abundance being a conjecture for higher dimensions, which is a key ingredient of the existence of log canonical models (c.f. \cite[Remark 13]{HP}), it is unclear whether $\mathcal F$ in the proof above extends to a functor from the moduli of stable odd hyperelliptic genus $g$ models to $\Lc_g$. If it extends, we expect the functor to be still fully faithful, as opposed to \cite[Remark 13]{HP} for birational transformations between semistable elliptic surfaces and stable elliptic curves over $\Pb^1$. The key obstruction on \cite[Remark 13]{HP}, assuming that the functor discussed in loc.cit. (which is an equivalence) extends, is that the essential surjectivity may not hold on the extension, whereas the functor from Theorem~\ref{thm:hypell_surf_and_qadm_fib} is not even essentially surjective to begin with.
    \end{rmk}
    
    \subsection{Hyperelliptic discriminant $\Delta_{g}$ of quasi--admissible hyperelliptic genus $g$ fibration}\label{subsec:disc}
    
    \par As we consider the algebraic surfaces $X$ as fibrations in genus $g$ curves over $\Pb^1$, the discriminant $\Delta_{g}(X)$ is the basic invariant of $X$. For the quasi--admissible hyperelliptic genus $g$ fibrations over $\Pb^1$, we have the work of \cite{Lockhart, Liu2} which describes the hyperelliptic discriminant $\Delta_{g}(X)$.

    \begin{defn} \cite[Definition 1.6, Proposition 1.10]{Lockhart}\label{Hyp_Disc}
        The hyperelliptic discriminant $\Delta_{g}$ of the monic odd--degree Weierstrass equation $y^2 = x^{2g+1} + a_{4}x^{2g-1} + a_{6}x^{2g-2} + a_{8}x^{2g-3} + \cdots + a_{4g+2}$ over a base field $K$ with $\mathrm{char}(K) \neq 2$ is 
    \[\Delta_{g} = 2^{4g} \cdot \text{Disc}(x^{2g+1} + a_{4}x^{2g-1} + a_{6}x^{2g-2} + \cdots + a_{4g+2})  \]
        which has $\deg(\Delta_{g}):=|\Delta_{g}| = 4g(2g+1)$ formally when we associate each variable $a_i$ with degree $i$.
    \end{defn}
    
    \par Note that when $g=1$, the discriminant $\Delta_1$ of the short Weierstrass equation $y^2=x^3+a_4x+a_6$ coincides with the usual discriminant $-16(4a_4^3-27a_6^2)$ of an elliptic curve. We can now formulate the moduli stack $\Lc_{g,|\Delta_{g}| \cdot n}$ of quasi--admissible fibration over $\Pb^1$ with a fixed discriminant degree $|\Delta_{g}| \cdot n = 4g(2g+1)n$ and a marked Weierstrass point :

    \begin{prop}\label{prop:qadm_fib}
        Assume $\mathrm{char}(K) = 0$ or $> 2g+1$. Then, the moduli stack $\Lc_{g,|\Delta_{g}| \cdot n}$ of quasi--admissible hyperelliptic genus $g$ fibrations over $\Pb^{1}_{K}$ with a marked Weierstrass point and a hyperelliptic discriminant of degree $|\Delta_{g}| \cdot n = 4g(2g+1)n$ over a base field $K$  is the tame Deligne--Mumford Hom stack $\Hom_n(\Pb^1,\Hc_{2g}[2g-1])$ parameterizing the $K$-morphisms $f:\Pb^1 \rightarrow \Hc_{2g}[2g-1]$ with $\Hc_{2g}[2g-1] \iso \Pc(\vec{\lambda_g}) = \Pc(4,6,8,\dotsc,4g+2)$ such that $f^*\Oc_{\Pc(\vec{\lambda}_g)}(1) \cong \Oc_{\Pb^1}(n)$ with $n \in \N$ .
    \end{prop}

    \begin{proof}
        Since $\deg f^*\Oc_{\Pc(\vec{\lambda}_g)}(1)=n$ is an open condition, $\Hom_n(\Pb^1,\Hc_{2g}[2g-1])$ is an open substack of $\Hom(\Pb^{1},\Hc_{2g}[2g-1])$. Now, it suffices to show that $\deg f=n$ (i.e., $\deg f^*\Oc_{\Pc(\vec{\lambda_g})}(1)=n$) if and only if the discriminant degree of the corresponding quasi--admissible fibration is $4g(2g+1)n$. Note that $\deg f=n$ if and only if the quasi--admissible fibration is given by the Weierstrass equation
        \[
        y^2=x^{2g+1}+a_4x^{2g-1}+a_6x^{2g-2}+\dotsb+a_{4g+2}
        \]
        where $a_i$'s are sections of $\Oc(in)$, since $a_i$'s represent the coordinates of $\Pc(4,6,\dotsc,4g+2)$. Then by Definition~\ref{Hyp_Disc}, it is straightforward to check that $\Delta_g$ has the discriminant degree $4g(2g+1)n$. 
    \end{proof}
    
    Now we are ready to count the number $|\Lc_{g,|\Delta_g|\cdot n}(\Fb_q)/\sim|$ of $\Fb_q$-isomorphism classes of quasi--admissible genus $g$ fibrations over $\Pb^1_{\Fb_q}$:
    
    \begin{proof}[Proof of Theorem~\ref{thm:low_genus_count}]
    By Proposition~\ref{prop:algo_ptcount}, for a fixed $g$, it suffices to understand when a connected component $\Hom_n(\Pb^1,\Pc((\vec\lambda_g)_r))$ (indexed by $r$) of $\Ic(\Lc_{g,|\Delta_g|\cdot n})$ is nonempty for $\vec\lambda_g = (4,6,\dotsc,4g+2)$; this is equivalent to finding $r$ such that the subtuple $(\vec\lambda_g)_r$ has length at least two. Table~\ref{tab:list_subtuple_qadm} describes all such possible $r$'s for given low values of $g=2,3,4$:
    
    \begin{table}
        \centering
        \begin{tabular}{|c|c|c|c|c|}
            \hline
            $r$ & $\varphi(r)$ & $g=2$ & $g=3$ & $g=4$\\
            \hline
            $1,2$ & 1 & $(4,6,8,10)$ & $(4,6,8,10,12,14)$ & $(4,6,8,10,12,14,16,18)$\\
            $3$ & 2 & -- & $(6,12)$ & $(6,12,18)$\\
            $4$ & 2 & $(4,8)$ & $(4,8,12)$ & $(4,8,12,16)$\\
            $6$ & 2 & -- & $(6,12)$ & $(6,12,18)$\\
            $8$ & 4 & -- & -- & $(8,16)$\\
            \hline
        \end{tabular}
        \caption{Table of all tuples $(\vec\lambda_g)_r$ of length at least two for low genus $g=2,3,4$. Entry has -- when $(\vec\lambda_g)_r$ has length zero or one.}
        \label{tab:list_subtuple_qadm}
    \end{table}
    
    Summing the weighted point counts of $\Hom$ stacks from Proposition~\ref{prop:motivecount_Hom} into Proposition~\ref{prop:algo_ptcount} gives the desired formula, where we use the division function $\delta(r,q-1)$ (defined in Theorem~\ref{thm:low_genus_count}) to indicate that we take the contribution of $\#_q(\Hom_n(\Pb^1,\Pc(\vec\lambda_r)))$ whenever $r \in R$ (i.e., $r$ divides $q-1$).
    
    The same method directly applies when $g \ge 5$.
    \end{proof}

    \section{Counting odd--degree hyperelliptic curves over $\mathbb{P}_{\mathbb{F}_q}^{1}$ by $\Delta_{g}$}\label{sec:globfield}

    \par We first recall the definition of a \textit{global field}. Let $S$ be the set of places of a field $K$ and $|\cdot|_v$ be the normalized absolute value for each place $v \in S$. 

    \begin{defn}\label{def:globfield}
    A field $K$ is a global field if all completions $K_v$ of $K$ at each place $v \in S$ is a local field, and $K$ satisfies the product formula $\prod\limits_{v} |\alpha|_v = 1$ over all places $v \in S$ for all $\alpha \in K^*$ . 
    \end{defn}

    \par And we have the fundamental Artin-Whaples Theorem proven in 1945 \cite{AW, AW2} which emphasized the close analogy between the theory of algebraic number fields and the theory of function fields of algebraic curves over finite fields. The axiomatic method used in these papers unified the two global fields from a valuation theoretic perspective by clarifying the role of the product formula.

    \begin{thm}[Artin-Whaples]\label{thm:Artin_Whaples}
    Every global field is a finite extension of $\Qb$ or $\Fb_q(t)$ .
    \end{thm}
    
    \par Focusing upon the global function fields $\Fb_q(t)$, we need to fix an affine chart $\A^1_{\Fb_q} \subset \Pb^1_{\Fb_q}$ and its corresponding ring of functions $\Fb_q[t]$ interpreted as the ring of integers of the field of fractions $\Fb_q(t)$ of $\Pb^1_{\Fb_q}$. This is necessary since $\Fb_q[t]$ could come from any affine chart of $\Pb^1_{\Fb_q}$, whereas the ring of integers $\Oc_K$ for the number field $K$ is canonically determined. We denote $\infty \in \Pb^1_{\Fb_q}$ to be the unique point not in the chosen affine chart.

    \medskip

    \par Note that for a maximal ideal $\pf$ in $\Oc_K$, the residue field $\Oc_K / \pf$ is finite. One could think of $\pf$ as a point in $\Spec~\Oc_K$ and define the \textit{height} of a point $\pf$.


    \begin{defn}\label{def:height_pt}
        Define the height of a point $\pf$ to be $ht(\pf) := |\Oc_K / \pf|$ the cardinality of the residue field $\Oc_K / \pf$. 
    \end{defn}

    \par We recall the notion of \textit{bad reduction} \& \textit{good reduction}:
    
    \begin{defn}\label{def:reductions}
    Let $C$ be an odd--degree hyperelliptic genus $g$ curve over $K$ given by the odd--degree Weierstrass equation 
    \[
    y^2=x^{2g+1}+a_4x^{2g-1}+a_6x^{2g-2}+\dotsb+a_{4g+2},
    \]
    with $a_{2i+2} \in \Oc_K$ for every $1 \le i \le 2g$. Then $C$ has bad reduction at $\pf$ if the fiber $C_{\pf}$ over $\pf$ is a singular curve of degree $2g+1$. The prime $\pf$ is said to be of good reduction if $C_{\pf}$ is a smooth hyperelliptic genus $g$ curve.
    \end{defn}
    
    \par Consider the case when $K=\Fb_q(t)$, and a quasi--admissible model $f:X \rightarrow \Pb^1_{\Fb_q}$ (a quasi--admissible fibration with smooth geometric generic fiber). For simplicity, assume that $X$ does not have a singular fiber over $\infty \in \Pb^1_{\Fb_q}$. Note that the primes $\pf_i$ of bad reductions of $f$ are precisely points of the discriminant divisor $\Delta_g(X)=\sum k_i \cdot \pf_i$, as the fiber $X_{\pf_i}$ is singular over $\Delta_{g}(X)$. When $K = \Fb_q(t)$ the global function field, we have $\Delta_{g}(X) \in H^0(\Pb^{1}, \Oc(4g(2g+1)n))$ by the proof of Proposition~\ref{prop:qadm_fib}. We can define the height of $\Delta_g(X)$ as follows: 
    
    \begin{defn}\label{def:height_Disc}
        The height $ht(\Delta_{g}(X))$ of a discriminant divisor $\Delta_{g}(X)$ in $\Pb^1_{\Fb_q}$ is $q^{\deg \Delta_{g}(X)}$. As a convention, if a divisor $\Delta_{g}(X)$ is given as a zero section of any line bundle, then set $ht(\Delta_{g}(X))=\infty$.
    \end{defn}
    In general, the height of a hyperelliptic discriminant $\Delta_{g}(X)$ of any $X$ (without nonsingular fiber assumption over $\infty$) is defined as $q^{|\Delta_{g}(X)|}$ where $\deg(\Delta_{g}(X)):=|\Delta_{g}(X)|$ is equal to $4g(2g+1)n$.




    \medskip
    
    \par As the hyperelliptic discriminant divisor $\Delta_{g}(X)$ is an invariant of a quasi--admissible model $f : X \to \Pb^1$, we count the number of $\Fb_q$--isomorphism classes of quasi--admissible hyperelliptic genus $g$ fibrations on the function field $\Fb_q(t)$ by the bounded height of $\Delta_{g}(X)$:

    \[\Zc_{g, \Fb_q(t)}(\Bc) := |\{\text{Quasi--admissible odd--degree hyperelliptic curves over } \Pb^1_{\Fb_q} \text{ with } 0<ht(\Delta_{g}) \le \Bc\}|\] 


    \medskip

    \par We now prove the sharp enumerations on $\Zc_{g,\Fb_q(t)}(\Bc)$ 

    %

    \begin{thm} [Sharp enumeration on $\Zc_{g, \Fb_q(t)}(\Bc)$] \label{thm:g_sharper_asymp_full}
        If $\text{char} (\Fb_q) > 2g+1$, then the function $\Zc_{g, \Fb_q(t)}(\Bc)$, which counts the number of quasi--admissible odd--degree hyperelliptic genus $g \ge 2$ curves $X$ over $\Pb^1_{\Fb_q}$ ordered by $0< ht(\Delta_{g}(X)) = q^{4g(2g+1)n} \le \Bc$, satisfies:
        \begingroup
        \allowdisplaybreaks
        \begin{align*}
         \Zc_{2, \Fb_q(t)}(\Bc) &\le 2 \frac{(q^{31} + q^{30} + q^{29} - q^{27} - q^{26} - q^{25})}{(q^{28}-1)} \cdot \left( \Bc^{\frac{7}{10}} - 1 \right) + 2\delta(4,q-1) \frac{(q^{13} - q^{11})}{(q^{12}-1)} \cdot \left( \Bc^{\frac{3}{10}} - 1 \right)\\\\
         \Zc_{3, \Fb_q(t)}(\Bc) &\le 2 \frac{(q^{59} + q^{58} + \dotsb + q^{55} - q^{53} - q^{52} - \dotsb - q^{49})}{(q^{54}-1)} \cdot \left( \Bc^{\frac{9}{14}} - 1 \right) \\
         &\phantom{\le}\text{ }+ 2\delta(4,q-1) \frac{(q^{26} + q^{25} - q^{23} - q^{22})}{(q^{24}-1)} \cdot \left( \Bc^{\frac{2}{7}} - 1 \right) + 4\delta(6,q-1) \frac{({q^{19} - q^{17}})}{(q^{18}-1)} \cdot \left( \Bc^{\frac{3}{14}} - 1 \right) \\\\
         \Zc_{4, \Fb_q(t)}(\Bc) &\le 2 \frac{(q^{95} + q^{94} + \dotsb + q^{89} - q^{87} - q^{52} - \dotsb - q^{81})}{(q^{88}-1)} \cdot \left( \Bc^{\frac{11}{18}} - 1 \right) \\
         &\phantom{\le}\text{ }+ 2\delta(4,q-1) \frac{(q^{43} + q^{42} + q^{41} - q^{39} - q^{38} - q^{37})}{(q^{40}-1)} \cdot \left( \Bc^{\frac{5}{18}} - 1 \right) \\
         &\phantom{\le}\text{ }+ 4\delta(6,q-1) \frac{(q^{38} + q^{37} - q^{35} - q^{34})}{(q^{36}-1)} \cdot \left( \Bc^{\frac{1}{4}} - 1 \right) + 4\delta(8,q-1) \frac{({q^{25} - q^{23}})}{(q^{24}-1)} \cdot \left( \Bc^{\frac{1}{6}} - 1 \right)
        \end{align*}
        \endgroup
        which is an equality when $\Bc = q^{4g(2g+1)n}$ with $n \in \mathbb{Z}_{\geq 1}$ implying that each upper bound is the sharp enumeration, i.e., the upper bound is equal to the function at infinitely many values of $\Bc \in \Zb_{\geq 1}$. 
    \end{thm} 

    \begin{proof}[Proof of Main Theorem~\ref{Mthm:g_sharper_asymp} for $g=2$]
        
        \par Knowing the number of $\Fb_q$-isomorphism classes of quasi--admissible genus 2 fibrations of discriminant degree $40n$ over $\Fb_q$ is \[|\Lc_{2,40n}(\Fb_q)/\sim| = 2 q^{28n} \cdot (q^{3} + q^{2} + q^{1} - q^{-1} - q^{-2} - q^{-3}) + 2 \delta(4,q-1) q^{12n} \cdot (q^{1}-q^{-1})\] 
        by Theorem~\ref{thm:low_genus_count}, we explicitly compute the upper bound for $\Zc_{2, \Fb_q(t)}(\Bc)$ as the following,
        
        \begingroup
        \allowdisplaybreaks
        \begin{align*}
        \Zc_{2, \Fb_q(t)}(\Bc) & = \sum \limits_{n=1}^{\left \lfloor \frac{log_q \Bc}{40} \right \rfloor} |\Lc_{2, 40n}(\Fb_q)/\sim| \\
        & = \sum \limits_{n=1}^{\left \lfloor \frac{log_q \Bc}{40} \right \rfloor} 2 \cdot q^{28n} \cdot (q^{3} + q^{2} + q^{1} - q^{-1} - q^{-2} - q^{-3}) + 2\delta(4,q-1)  q^{12n} \cdot (q^{1}-q^{-1}) \\
        & = 2 \cdot (q^{3} + q^{2} + q^{1} - q^{-1} - q^{-2} - q^{-3}) \sum \limits_{n=1}^{\left\lfloor\frac{log_q \Bc}{40}\right\rfloor} q^{28n} + 2\delta(4,q-1) ({q^{1} - q^{-1}}) \sum \limits_{n=1}^{\left\lfloor\frac{log_q \Bc}{40}\right\rfloor} q^{12n}\\
        & \le 2 \cdot (q^{3} + q^{2} + q^{1} - q^{-1} - q^{-2} - q^{-3}) \left( q^{28} + \cdots + q^{28 \cdot (\frac{log_q \Bc}{40})}\right) \\
        & \phantom{=  }\; + 2\delta(4,q-1) ({q^{1} - q^{-1}}) \left( q^{12} + \cdots + q^{12 \cdot (\frac{log_q \Bc}{40})}\right) \\
        & = 2 \cdot (q^{3} + q^{2} + q^{1} - q^{-1} - q^{-2} - q^{-3}) \left( \frac{q^{28} \cdot ( \Bc^{\frac{7}{10}} - 1) }{(q^{28}-1)} \right) \\
        &\phantom{= }\; + 2 \delta(4,q-1) ({q^{1} - q^{-1}}) \left( \frac{q^{12} \cdot ( \Bc^{\frac{3}{10}} - 1) }{(q^{12}-1)} \right)  \\
        & = 2 \frac{(q^{31} + q^{30} + q^{29} - q^{27} - q^{26} - q^{25})}{(q^{28}-1)} \cdot \left( \Bc^{\frac{7}{10}} - 1 \right) + 2 \delta(4,q-1) \frac{(q^{13} - q^{11})}{(q^{12}-1)} \cdot \left( \Bc^{\frac{3}{10}} - 1 \right)
        \end{align*}
        \endgroup

        On the fourth line of the equations above, inequality becomes an equality if and only if $n:= \frac{log_q \Bc}{40} \in \N$, i.e., $\Bc=q^{40n}$ with $n \in \N$. This implies that the acquired upper bound on $\Zc_{2, \Fb_q(t)}(\Bc)$ is the sharp enumeration, i.e., the upper bound is equal to the function at infinitely many values of $\Bc \in \Zb_{\geq 1}$.
    \end{proof}

    \par As there are non-hyperelliptic curves for higher genus $g \ge 3$ curves, $\Zc_{g \ge 3, \Fb_q(t)}(\Bc)$ counts the quasi--admissible hyperelliptic genus $g \ge 3$ curves over $\Pb^{1}_{\Fb_q}$ only. We determine $\Zc_{3, \Fb_q(t)}(\Bc)$ explicitly thereby counting the quasi--admissible hyperelliptic genus 3 curves over $\Pb^{1}_{\Fb_q}$.

    \begin{proof}[Proof of Main Theorem~\ref{Mthm:g_sharper_asymp} for $g=3$]
        Knowing the number of $\Fb_q$-isomorphism classes of quasi--admissible hyperelliptic genus 3 fibrations of discriminant degree $84n$ over $\Fb_q$ is $\left|\Lc_{3,84n}(\Fb_q)/\sim\right| = 2 \cdot q^{54n} \cdot (q^{5} +  \dotsb + q^{1} - q^{-1}  - \dotsb - q^{-5}) + 2\delta(4,q-1) q^{24n} \cdot (q^{2} + q^{1} - q^{-1} - q^{-2}) + 4\delta(6,q-1) q^{18n} \cdot (q^{1} - q^{-1})$ by Theorem~\ref{thm:low_genus_count}, we explicitly compute the upper bound for $\Zc_{3, \Fb_q(t)}(\Bc)$ similarly as genus 2 case.
    \end{proof}

    \par We conclude with $\Zc_{4, \Fb_q(t)}(\Bc)$ counting the quasi--admissible hyperelliptic genus 4 curves over $\Pb^{1}_{\Fb_q}$.

    \begin{proof}[Proof of Main Theorem~\ref{Mthm:g_sharper_asymp} for $g=4$]
        Knowing the number of $\Fb_q$-isomorphism classes of quasi--admissible hyperelliptic genus 4 fibrations of discriminant degree $144n$ over $\Fb_q$ is $\left|\Lc_{4,144n}(\Fb_q)/\sim\right| = 2 \cdot q^{88n} \cdot (q^{7} + \dotsb + q^{1} - q^{-1} - \dotsb - q^{-7}) + 2 \delta(4,q-1) q^{40n} \cdot (q^{3} + q^{2} + q^{1} - q^{-1} - q^{-2} - q^{-3}) + 4 \delta(6,q-1) q^{36n} \cdot (q^{2} + q^{1} - q^{-1} - q^{-2}) + 4 \delta(8,q-1) q^{24n} \cdot (q^{1} - q^{-1})$ by Theorem~\ref{thm:low_genus_count}, we explicitly compute the upper bound for $\Zc_{4, \Fb_q(t)}(\Bc)$ similarly as genus 2 case.
    \end{proof}

    \par The computation of sharp enumerations with precise lower order terms for the higher genus cases $\Zc_{g, \Fb_q(t)}(\Bc)$ can be done similarly after working out $\left|\Lc_{g, |\Delta_{g}| \cdot n}(\Fb_q)/\sim\right|$ by the Theorem~\ref{thm:low_genus_count}. While the lower order terms vary, the order of the leading term can be found by the following. Then, we define the similar counting function $\Zc'_{g, \Fb_q(t)}(\Bc)$ as
    \[\Zc'_{g, \Fb_q(t)}(\Bc) := |\{\text{Stable odd--degree hyperelliptic curves over } \Pb^1_{\Fb_q} \text{ with } 0<ht(\Delta_{g}) \le \Bc\}|\] 

    
    Below, we prove the following closed-form upper bound for $\Zc'_{g, \Fb_q(t)}(\Bc)$:

    \begin{prop}[Estimate on $\Zc'_{g, \Fb_q(t)}(\Bc)$] \label{prop:curve_count} 
        If $\text{char} (\Fb_q) > 2g+1$, then the function $\Zc'_{g, \Fb_q(t)}(\Bc)$, which counts the number of stable odd--degree hyperelliptic curves $X$ with a marked Weierstrass point over $\Pb^1_{\Fb_q}$ ordered by $0< ht(\Delta_{g}(X)) = q^{4g(2g+1)n} \le \Bc$, satisfies:
        \[
            \Zc'_{g, \Fb_q(t)}(\Bc) \le 2g \cdot \frac{q^{4g(g+1)+1} \cdot (q^{2g-1}-1)(q^{2g}-1)}{(q-1)(q^{2g(2g+3)}-1)} \cdot \left( \Bc^{\frac{2g+3}{4g+2}}-1 \right) \]
    \end{prop}

    \begin{proof}
         
        \par Note that the automorphism group of minimum order of $\varphi_g$ is the generic stabilizer group $\mu_{\delta} = \mu_2$ of $\Pc(\vec{\lambda_g})$ and the automorphism group of maximum order of $\varphi_g$ is $\mu_{\omega} = \mu_{2g}$ as $2g$ is the maximum value of GCD for all possible pairs among $\vec{\lambda_g} = (4,6,8,\dotsc,4g+2)$. By Corollary~\ref{cor:maincount} we know that the number of $\Fb_q$-isomorphism classes of quasi--admissible hyperelliptic genus $g$ fibrations of hyperelliptic discriminant degree $|\Delta_{g}| \cdot n=4g(2g+1)n$ over $\Fb_q$ is $2 \cdot \#_q\left(\Lc_{g, |\Delta_{g}| \cdot n}\right) \le \left|\Lc_{g, |\Delta_{g}| \cdot n}(\Fb_q)/\sim\right| \le 2g \cdot \#_q\left(\Lc_{g, |\Delta_{g}| \cdot n}\right)$, we can explicitly compute the upper bound for $\Zc'_{g, \Fb_q(t)}(\Bc)$ as the following,
        
        \begingroup
        \allowdisplaybreaks
        \begin{align*}
        \Zc'_{g, \Fb_q(t)}(\Bc) & = \sum \limits_{n=1}^{\left \lfloor \frac{log_q \Bc}{4g(2g+1)} \right \rfloor} |\Lc_{g, |\Delta_{g}| \cdot n}(\Fb_q)/\sim| \\
        & \le \sum \limits_{n=1}^{\left \lfloor \frac{log_q \Bc}{4g(2g+1)} \right \rfloor} 2g \cdot q^{2g(2g+3)n} \cdot \frac{(q^{2g-1}-1)(q-q^{-2g+1})}{q-1} \\
        & = 2g \cdot \frac{(q^{2g-1}-1)(q-q^{-2g+1})}{q-1} \sum \limits_{n=1}^{\left\lfloor\frac{log_q \Bc}{4g(2g+1)}\right\rfloor} q^{2g(2g+3)n} \\
        & = 2g \cdot \frac{(q^{2g-1}-1)(q-q^{-2g+1})}{q-1} \left( q^{2g(2g+3)} + \cdots + q^{2g(2g+3) \cdot (\frac{log_q \Bc}{4g(2g+1)})}\right)\\
        & = 2g \cdot \frac{(q^{2g-1}-1)(q-q^{-2g+1})}{q-1} \cdot \left( \frac{q^{2g(2g+3)} \cdot (\Bc^{\frac{2g+3}{4g+2}}-1)}{q^{2g(2g+3)}-1} \right)\\
        & = 2g \cdot \frac{q^{4g(g+1)+1} \cdot (q^{2g-1}-1)(q^{2g}-1)}{(q-1)(q^{2g(2g+3)}-1)} \cdot \left( \Bc^{\frac{2g+3}{4g+2}}-1 \right)
        \end{align*}
        \endgroup
        
        This implies that this upper bound has the leading term of order $\Oc_q\left( \Bc^{\frac{2g+3}{4g+2}}\right)$ where $\mathcal{O}_q$-constant is an explicit rational function of $q$ with the corresponding lower order terms for each genus $g \ge 2$.
    \end{proof}

    \par For higher genus $g \ge 3$, we count the Jacobians of hyperelliptic curves.

    \begin{thm} [Estimate on $\Nc_{g, \Fb_q(t)}(\Bc)$]\label{thm:Hyp_Jac_Fqt}
    If $\text{char} (\Fb_q) > 2g+1$, then the function $\Nc_{g, \Fb_q(t)}(\Bc)$, which counts the number of principally polarized hyperelliptic Jacobians $A = \mathrm{Jac}(X)$ where $X$ is a stable hyperelliptic genus $g \ge 3$ curves $X$ with a marked Weierstrass point over $\Pb^1_{\Fb_q}$ ordered by $0< ht(\Delta_{g}(X)) = q^{4g(2g+1)n} \le \Bc$, satisfies:
    \[
            \Nc_{g, \Fb_q(t)}(\Bc) \le 2g \cdot \frac{q^{4g(g+1)+1} \cdot (q^{2g-1}-1)(q^{2g}-1)}{(q-1)(q^{2g(2g+3)}-1)} \cdot \left( \Bc^{\frac{2g+3}{4g+2}}-1 \right) 
    \]
    \end{thm}

    \begin{proof} 
    Proposition~\ref{prop:curve_count} provides an upper bound on the number of stable hyperelliptic genus $g \ge 3$ curves with a marked Weierstrass point over $\Pb^1_{\Fb_q}$ with $\text{char} (\Fb_q) > 2g+1$. The upper bound follows from the open immersion property of the Torelli map $\tau_g$ restricted to the hyperelliptic locus (c.f. \cite[Theorem 1.2.]{Landesman}). 
    \end{proof}

    \medskip



    \subsection*{Connection to other Heights and related Enumerations}
    \par On a related note, we recall that the number of discriminants $\Delta_1$ of an elliptic curve over $\Zb$ with smooth generic fiber such that $\Delta_1 \le \Bc$ is estimated to be asymptotic to $\Oc\left(\Bc^{\frac{5}{6}}\right)$ by \cite{BMc}. The lower order term of order $\Oc\left(\Bc^{(7-\frac{5}{27}+\epsilon)/12}\right)$ for counting the stable elliptic curves over $\Qb$ by the bounded height of squarefree $\Delta_1$ was suggested by the work of \cite{Baier} improving upon their previous error term in \cite{BB}. In fact, Baier proved his asymptotic under the assumption of the generalized Riemann hypothesis with the twelveth root of the na\"ive height function on elliptic curves which gives the prediction above. For global function fields $\Fb_q(t)$, by considering the moduli of semistable elliptic surfaces and finding its motive/point count, we acquire the sharp enumeration \cite[Theorem 3]{HP} on $\Zc_{1, \Fb_q(t)}$ for counting the semistable elliptic curves by the bounded height of $\Delta_1(X)$ over $\Pb^{1}_{\Fb_q}$ with $\text{char} (\Fb_q) \neq 2,3$ giving the leading term of order $\Oc_{q}\left(\Bc^{\frac{5}{6}}\right)$ and the lower order term of zeroth order $\Oc_{q}(1)$. The arithmetic invariant which leads to the above counting also has been established in the past via different method by the seminal work of \cite{dJ}, which works also in characteristic 2 and 3. 

    \medskip


    \par For genus $g \ge 2$ hyperelliptic curves, we have the qualitative finiteness shown by the classical works of \cite{Parshin, Oort}. For effective results we have \cite{BG} for counting by na\"ive height and \cite{Kanel} for partly explicit upper bound. 

    \medskip

    \par Our project could be considered as an extension of the beautiful work done in \cite{EVW} by Jordan S. Ellenberg, Akshay Venkatesh and Craig Westerland. They proved in loc.cit. a function field analogue of the Cohen-Lenstra heuristics on distributions of class groups by point counting the \textit{Hurwitz spaces} parametrizing branched covers of the complex projective line. As the branched covers of the $\Pb^{1}$ are the fibrations with $0$-dimensional fibers, the moduli of fibrations $f: X \to \Pb^{1}$ on fibered surfaces $X$ with $1$-dimensional fibers is the next most natural case to work on. The counting technique in our project is driven largely by the inspiring work of Benson Farb and Jesse Wolfson \cite{FW} which in turn was motivated by the ideas in Graeme Segal's classical paper \cite{Segal}.




    \medskip
    \medskip
    \bigskip


    \textbf{Acknowledgments.} 
    The authors are indebted to Ariyan Javanpeykar for explaining to us the open immersion property of the Torelli map restricted to the hyperelliptic locus. We also thank Brian Conrad, Minhyong Kim and Qing Liu for useful comments on a draft of this paper as well as Dori Bejleri, Joe Harris, Will Sawin, Jesse Wolfson and Craig Westerland for helpful discussions. Jun-Yong Park was supported by the Max Planck Institute for Mathematics and the Institute for Basic Science in Korea (IBS-R003-D1) and thanks the Center for Geometry and Physics for its hospitality. Changho Han acknowledges the partial support of the Natural Sciences and Engineering Research Council of Canada (NSERC), [PGSD3-487436-2016], and of the Research and Training Group in Algebraic Geometry, Algebra and Number Theory, at the University of Georgia.

    \medskip


    \vspace{+16 pt}

    \noindent Changho Han \\
    \textsc{Department of Mathematics, University of Georgia, Athens, GA 30602, USA}\\
      \textit{E-mail address}: \texttt{Changho.Han@uga.edu}

    \vspace{+16 pt}

    \noindent Jun--Yong Park \\
    \textsc{Max-Planck-Institut f\"ur Mathematik, Vivatsgasse 7 \\ 53111 Bonn, Germany} \\
      \textit{E-mail address}: \texttt{junepark@mpim-bonn.mpg.de}


\begin{thebibliography}{99}

        \bibitem[AGV]{AGV} D. Abramovich, T. Graber and A. Vistoli, 
        \textit{Gromov-Witten Theory of Deligne-Mumford Stacks}, 
        American Journal of Mathematics, \textbf{130}, No. 5, (2008): 1337--1398.

        \bibitem[AOV]{AOV} D. Abramovich, M. Olsson and A. Vistoli, 
        \textit{Tame stacks in positive characteristic}, 
        Annales de l'Institut Fourier, \textbf{58}, No. 4, (2008): 1057--1091.


        \bibitem[AW]{AW} E. Artin and G. Whaples,
        \textit{Axiomatic characterization of fields by the product formula for valuations},
        Bulletin of the American Mathematical Society, \textbf{51}, No. 7, (1945): 469--492.

        \bibitem[AW2]{AW2} E. Artin and G. Whaples,
        \textit{A note on axiomatic characterization of fields},
        Bulletin of the American Mathematical Society, \textbf{52}, No. 4, (1946): 245--247.

        \bibitem[Baier]{Baier} S. Baier,
        \textit{Elliptic curves with square-free $\Delta$},
        International Journal of Number Theory, \textbf{12}, No. 3, (2016): 737--764.

        \bibitem[BB]{BB} S. Baier and T. D. Browning,
        \textit{Inhomogeneous cubic congruences and rational points on Del Pezzo surfaces},
        Journal f\"ur die reine und angewandte Mathematik, \textbf{680}, (2013): 69--151.
        
        \bibitem[Behrend]{Behrend} K. A. Behrend,
        \textit{The Lefschetz trace formula for algebraic stacks},
        Inventiones mathematicae, \textbf{112}, No. 1, (1993): 127--149.


        \bibitem[BG]{BG} M. Bhargava and B. Gross,
        \textit{The average size of the 2-Selmer group of Jacobians of hyperelliptic curves having a rational Weierstrass point},
        Automorphic representations and L-functions, The Studies in Mathematics, \textbf{22}, Tata Institute of Fundamental Research, Mumbai (2013): 23–-91.


        \bibitem[BMc]{BMc} A. Brumer and O. McGuinness,
        \textit{The behavior of the Mordell-Weil group of elliptic curves},
        Bulletin of the American Mathematical Society, \textbf{23}, No. 2, (1990): 375--382.




        \bibitem[de Jong]{dJ} A. J. de Jong,
        \textit{Counting elliptic surfaces over finite fields},
        Moscow Mathematical Journal, \textbf{2}, No. 2, (2002): 281--311.

        \bibitem[DM]{DM} P. Deligne and D. Mumford,
        \textit{The irreducibility of the space of curves of given genus},
        Publications Math\'ematiques de l'I.H.\'E.S., \textbf{36}, (1969): 75--109.



        \bibitem[EH]{EH} D. Eisenbud and J. Harris,
        \textit{The Kodaira dimension of the moduli space of curves of genus $\ge$ 23},
        Inventiones mathematicae, \textbf{90}, No. 2, (1987): 359--387.

        \bibitem[Ekedahl]{Ekedahl} T. Ekedahl,
        \textit{The Grothendieck group of algebraic stacks},
        arXiv:0903.3143, (2009).

        \bibitem[EVW]{EVW} J. Ellenberg, A. Venkatesh and C. Westerland,
        \textit{Homological stability for Hurwitz spaces and the Cohen-Lenstra conjecture over function fields},
        Annals of Mathematics, \textbf{183}, No. 3, (2016): 729--786.


        \bibitem[Fedorchuk]{Fedorchuk} M. Fedorchuk, 
        \textit{Moduli spaces of hyperelliptic curves with A and D singularities}, 
        Mathematische Zeitschrift, \textbf{276}, No. 1-2, (2014): 299--328.



        \bibitem[Fujino]{Fujino} O. Fujino,
        \textit{Minimal Model Theory for Log Surfaces}, 
        Publications of the Research Institute for Mathematical Sciences, \textbf{48}, No. 2, (2012): 339--371.

        \bibitem[FW]{FW} B. Farb and J. Wolfson,
        \textit{Topology and arithmetic of resultants, I}, 
        New York Journal of Mathematics, \textbf{22}, (2016): 801--821.

        \bibitem[GGW]{GGW} C. G\'omez-Gonz\'ales and J. Wolfson,
        \textit{Problems in Arithmetic Topology},
        Research in the Mathematical Sciences, Special Issue of PIMS 2019 Workshop on Arithmetic Topology, \textbf{8}, No. 2, \#23 (2021).
     
        \bibitem[GHM]{GHM} C. Galindo, F. Hernando and F. Monserrat,
        \textit{The log-canonical threshold of a plane curve},
        Mathematical Proceedings of the Cambridge Philosophical Society, \textbf{160}, No. 3, (2016): 513--535.

        
        \bibitem[Hassett]{Hassett} B. Hassett,
        \textit{Moduli spaces of weighted pointed stable curves},
        Advances in Mathematics, \textbf{173}, No. 2, (2003): 316--352.

        \bibitem[Hassett2]{Hassett2} B. Hassett,
        \textit{Classical and minimal models of the moduli space of curves of genus two},
        Geometric Methods in Algebra and Number Theory, Progress in Mathematics, \textbf{235}, Birkh\"auser Boston, (2005): 169–-192.

        \bibitem[HM]{HM} J. Harris and D. Mumford,
        \textit{On the Kodaira dimension of the moduli space of curves},
        Inventiones mathematicae, \textbf{67}, No. 1, (1982): 23--86.



        \bibitem[HP]{HP} C. Han and J. Park,
        \textit{Arithmetic of the moduli of semistable elliptic surfaces}, 
        Mathematische Annalen, \textbf{375}, No. 3--4, (2019): 1745--1760.

        
        \bibitem[J\"arvilehto]{Jarvilehto} T. J\"arvilehto,
        \textit{Jumping numbers of a simple complete ideal in a two-dimensional regular local ring},
        Memoirs of the American Mathematical Society, \textbf{214}, No. 1009, (2011).

        \bibitem[K\"anel]{Kanel} R. V. K\"anel,
        \textit{An effective proof of the hyperelliptic Shafarevich conjecture}, 
        Journal de Th\'eorie des Nombres de Bordeaux, \textbf{26}, No. 2, (2014): 507--530.

        
        \bibitem[Knudsen]{Knudsen} F. F. Knudsen,
        \textit{The projectivity of the moduli space of stable curves II, III},
        Mathematica Scandinavica, \textbf{52}, No. 2, (1983): 161--199, 200--212.
        
        \bibitem[Knutson]{Knutson} D. Knutson,
        \textit{Algebraic Spaces},
        Lecture Notes in Mathematics, \textbf{203}, Springer-Verlag, Berlin (1971).


        \bibitem[Landesman]{Landesman} A. Landesman,
        \textit{The Torelli map restricted to the hyperelliptic locus},
        Transactions of the American Mathematical Society, Series B, \textbf{8}, (2021): 354--378.

        \bibitem[Liedtke]{Liedtke} C. Liedtke,
        \textit{Algebraic Surfaces in Positive Characteristic}, 
        Birational Geometry, Rational Curves, and Arithmetic, Simons Symposia, Springer--Verlag New York (2013): 229--292.

        \bibitem[Liu]{Liu} Q. Liu,
        \textit{Algebraic geometry and arithmetic curves},
        Oxford Graduate Texts in Mathematics, \textbf{6}, Oxford University Press (2002).

        \bibitem[Liu2]{Liu2} Q. Liu,
        \textit{Mod\`eles entiers des courbes hyperelliptiques sur un corps de valuation discr\`ete},
        Transactions of the American Mathematical Society, \textbf{348}, No. 11, (1996): 4577--4610.


        \bibitem[LO]{LO} Y. Laszlo and M. Olsson
        \textit{The six operations for sheaves on Artin stacks II: Adic coefficients}, 
        Publications Math\'ematiques de l'I.H.\'E.S., \textbf{107}, (2008): 169--210.

        \bibitem[Lockhart]{Lockhart} P. Lockhart,
        \textit{On the discriminant of a hyperelliptic curve},
        Transactions of the American Mathematical Society, \textbf{342}, No. 2, (1994): 729--752.



        
        \bibitem[Milne]{Milne} J. S. Milne,
        \textit{Jacobian Varieties}, 
        Arithmetic Geometry, (Storrs, Conn., 1984), Springer, New York, NY, (1986): 167--212.

        

        \bibitem[Olsson]{Olsson} M. Olsson, 
        \textit{\underline{Hom}-stacks and restriction of scalars}, 
        Duke Mathematical Journal, \textbf{134}, No. 1, (2006): 139--164.

        \bibitem[Olsson2]{Olsson2} M. Olsson, 
        \textit{Algebraic Spaces and Stacks},
        Colloquium Publications, \textbf{62}, American Mathematical Society (2016).

        \bibitem[Oort]{Oort} F. Oort, 
        \textit{Hyperelliptic curves over number fields}, 
        Classification of Algebraic Varieties and Compact Complex Manifolds, Lecture Notes in Mathematics, \textbf{412}, Springer--Verlag Berlin Heidelberg (1974): 211--218.

        \bibitem[OS]{OS} F. Oort and J. Steenbrink,
        \textit{The local Torelli problem for algebraic curves},
        Journ\'ees de G\'eometrie Alg\'ebrique d'Angers, Juillet 1979 / Algebraic Geometry, Angers, Sijthoff \& Noordhoff. Alphen aan den Rijn-Germantown Md. 1980, (1979): 157--204.

        \bibitem[OU]{OU} F. Oort and K. Ueno, 
        \textit{Principally polarized abelian varieties of dimension two or three are Jacobian varieties},
        Journal of the Faculty of Science, University of Tokyo, Section IA Mathematics, \textbf{20}, (1973): 377--381.



        \bibitem[Parshin]{Parshin} A. N. Parshin,
        \textit{Minimal models of curves of genus 2 and homomorphisms of abelian varieties defined over a field of finite characteristic},
        Izv. Akad. Nauk SSSR Ser. Mat., \textbf{36}, No. 1, (1972): 67--109. (Russian).
        Mathematics of the USSR-Izvestiya, \textbf{6}, No. 1, (1972): 65--108. (English).

        \bibitem[PS]{PS} J. Park and H. Spink,
        \textit{Motive of the moduli stack of rational curves on a weighted projective stack},
        Research in the Mathematical Sciences, Special Issue of PIMS 2019 Workshop on Arithmetic Topology, \textbf{8}, No. 1, \#1 (2021).



        \bibitem[Segal]{Segal} G. Segal, 
        \textit{The topology of spaces of rational functions}, 
        Acta Mathematica, \textbf{143}, (1979): 39--72.




        \bibitem[Stankova]{Stankova} Z. Stankova, 
        \textit{Moduli of trigonal curves}, 
        Journal of Algebraic Geometry, \textbf{9}, No. 4, (2000): 607--662.
        
        \bibitem[Stacks]{Stacks} The Stacks Project Authors: Stacks Project (2021). \url{https://stacks.math.columbia.edu/tags}


        \bibitem[Sun]{Sun} S. Sun, 
        \textit{L-series of Artin stacks over finite fields}, 
        Algebra \& Number Theory, \textbf{6}, No. 1, (2012): 47--122.


        \bibitem[Tanaka]{Tanaka} H. Tanaka,
        \textit{Minimal Models and Abundance for Positive Characteristic Log Surfaces}, 
        Nagoya Mathematical Journal, \textbf{216}, (2014): 1--70.


        \bibitem[Weil]{Weil} A. Weil, 
        \textit{Zum Beweis des Torellischen Satzes},
        Nachrichten der Akademie der Wissenschaften in G\"ottingen. II. Mathematisch-Physikalische Klasse, (1957): 33--53.


    \end{thebibliography}
\end{document}